\documentclass{rsepublic}

\usepackage[utf8]{inputenc}
\usepackage[T1]{fontenc}

\usepackage{siunitx} %
\usepackage{amsfonts}
\usepackage{amsmath}
\usepackage{amssymb}
\usepackage{mathtools}
\usepackage{ragged2e}
\usepackage{float}
\restylefloat{table}

\Year{2017}
\Volume{}
\Paper{}
\Pagerange{\pageref{firstpage}--\pageref{lastpage}}
\begin{document}

\title[L-G expansions of exponential form]{Liouville-Green expansions of exponential form, with an application to modified Bessel functions}

\author[T. M. Dunster]{\textbf{T. M. Dunster}
\\
Department of Mathematics and Statistics, San Diego State University, San Diego, CA 92182-7720, USA.}
\MSdates[‘Received date’]{‘Accepted date’}
\label{firstpage}
\maketitle

\begin{abstract}
\justify
Linear second order differential equations of the form $d^{2}w/dz^{2}-\left 
\{ {u^{2}f\left( u,z\right) +g\left( z\right) }\right\} w=0$ are studied, where $\left\vert u\right\vert \rightarrow \infty $ and $z$ lies in a complex bounded or unbounded domain $\mathbf{D}$. If $f\left( u,z\right) $ and $g\left( z\right) $ are meromorphic in $\mathbf{D}$, and $f\left(u,z\right) $ has no zeros, the classical Liouville-Green/WKBJ approximation provides asymptotic expansions involving the exponential function. The coefficients in these expansions either multiply the exponential, or in an alternative form appear in the exponent. The latter case has applications to the simplification of turning point expansions as well as certain quantum mechanic problems, and new computable error bounds are derived. It is shown how these bounds can be sharpened to provide realistic error estimates, and this is illustrated by an application to modified Bessel functions of complex argument and large positive order. Explicit computable error bounds are also derived for asymptotic expansions for particular solutions of the nonhomogeneous equations of the form $d^{2}w/dz^{2}-\left\{ {u^{2}f\left(z\right) +g\left( z\right) }\right\} w=p\left( z\right)$.

\end{abstract}

\section{Introduction and main results}

\label{sec1} In this paper we study the classical problem of obtaining asymptotic expansions for solutions of differential equations of the form%
\begin{equation}
d^{2}w/dz^{2}=\left\{ {u^{2}f\left( u,z\right) +g\left( z\right) }\right\} w.
\label{eq1}
\end{equation}%
Here $u$ is a large parameter, real or complex, and $z$ lies in a complex domain $\mathbf{D}$ which may be unbounded. We assume that $f\left(u,z\right) $ and $g\left( z\right) $ are meromorphic in $\mathbf{D}$, and $f\left( u,z\right) $ has no zeros in this domain. We shall first assume that case ${f\left( u,z\right) =f\left( z\right) }$ is independent of $u$, and in section \ref{sec3} we consider the more general case where $f$ depends on $u$.

For the present case we make the standard Liouville transformation (see \cite[chap. 10]{Olver:1997}) 
\begin{equation}
d\xi /dz=f^{1/2}\left( z\right) ,  \label{eq2}
\end{equation}%
which upon integration gives the relation 
\begin{equation}
\xi =\int {f^{1/2}\left( z\right) dz.}  \label{eq3}
\end{equation}%
With the associated change of dependent variable%
\begin{equation}
W=f^{1/4}\left( z\right) w  \label{eq4}
\end{equation}%
the differential equation (\ref{eq1}) is transformed to 
\begin{equation}
d^{2}W/d\xi ^{2}=\left\{ {u^{2}+\psi \left( \xi \right) }\right\} W,
\label{eq5}
\end{equation}%
where 
\begin{equation}
\psi \left( \xi \right) =\frac{4f\left( z\right) {f}^{\prime \prime }\left(
z\right) -5{f}^{\prime 2}\left( z\right) }{16f^{3}\left( z\right) }+\frac{g\left( z\right) }{f\left( z\right) }.  \label{eq6}
\end{equation}%
This function is analytic in a $\xi$-domain $\Delta $ (say), which corresponds to some subset of the $z$-domain $\mathbf{D}$.

Classical Liouville-Green (WKBJ) asymptotic expansions for solutions of the new differential equation (\ref{eq5}) are of the form (see, for example, \cite[chap. 10]{Olver:1997})
\begin{equation}
W\sim \exp \left\{ \pm u\xi \right\} \sum\limits_{s=0}^{\infty }\left( \pm
1\right) ^{s}{\frac{A_{s}\left( \xi \right) }{u^{s}},}  \label{eq6a}
\end{equation}%
where $A_{0}\left( \xi \right) =1$ and%
\begin{equation}
A_{s+1}(\xi )=-\tfrac{1}{2}A_{s}^{\prime }(\xi )+\tfrac{1}{2}\int \psi (\xi
)A_{s}(\xi)d\xi \ \left( s=0,1,2,\cdots \right) .  \label{eq6b}
\end{equation}%
We remark that these involve repeated integration, which can lead to difficulties in their computation.

In this paper we study asymptotic expansions of solutions in the alternative form (see \cite[chap. 10, ex. 2.1]{Olver:1997})%
\begin{equation}
W\sim \exp \left\{ \pm u\xi +\sum\limits_{s=0}^{\infty }\left( \pm 1\right)
^{s}{\frac{E_{s}\left( \xi \right) }{u^{s}}}\right\} .  \label{eq6c}
\end{equation}%
These expansions are central in a new method developed in \cite{DGS:2017} on the computation of
asymptotic expansions for turning point problems, and also are important in certain quantum mechanic problems \cite{Dunster:1998}. 

Here the coefficients are found by substitution of (\ref{eq6c}) into (\ref{eq5}) and equating like powers of $u$. By this we find that
\begin{equation}
E_{s}\left( \xi \right) =\int {F_{s}\left( \xi \right) d\xi }\ \left( {s=1,2,3,\cdots }\right) ,  \label{eq10}
\end{equation}%
where 
\begin{equation}
F_{1}\left( \xi \right) ={\tfrac{1}{2}}\psi \left( \xi \right) ,\
F_{2}\left( \xi \right) =-{\tfrac{1}{4}}{\psi }^{\prime }\left( \xi \right) ,
\label{eq11}
\end{equation}%
and 
\begin{equation}
F_{s+1}\left( \xi \right) =-{\dfrac{1}{2}}\frac{dF_{s}\left( \xi \right) }{d\xi }-{\dfrac{1}{2}}\sum\limits_{j=1}^{s-1}{F_{j}\left( \xi \right)
F_{s-j}\left( \xi \right) }\ \left( {s=2,3,4\cdots }\right) .  \label{eq12}
\end{equation}

We remark that an alternative method of generating the same formal series solution is given in \cite{Moriguchi:1959}, using repeated applications of the Liouville transformation (\ref{eq3}) and (\ref{eq4}) to the differential equation (\ref{eq1}).

Comparing the coefficients given by (\ref{eq10}) to those of (\ref{eq6b}) we see an advantage that nested integrations are avoided. Furthermore, from \cite{DGS:2017} it was shown that the even coefficients $E_{2j}\left( \xi \right) 
$ ($j=1,2,3,\cdots $) can be evaluated without resorting to integration. Specifically, we can use the formal relation (which comes from Abel's
theorem)%
\begin{equation}
\left\{ u+\sum_{j=0}^{\infty }\frac{{F_{2j+1}(\xi )}}{{u^{2j+1}}}\right\}
\exp \left\{ 2\sum_{j=1}^{\infty }\frac{{E_{2j}(\xi )}}{{u^{2j}}}\right\}
\sim C,  \label{wronsk}
\end{equation}%
where $C$ is an arbitrary constant. We then asymptotically expand the LHS of this relation in inverse powers of $u$, and equate the coefficient of each term to a constant. As a result, the even coefficients $E_{2j}(\xi )$ ($j=1,2,3,\cdots $) are explicitly given in terms of $F_{2k+1}(\xi )$ ($k=0,1,2,\cdots ,j-1$) (which in turn are given by (\ref{eq11}) and (\ref{eq12})). For example (to within an arbitrary additive constant), $E_{2}=-\tfrac{1}{2}{F_{1}}$, $E_{4}=-\tfrac{1}{2}{F_{3}-}E_{2}{F_{1}-}E_{2}^{2}={\tfrac{1}{4}F_{1}^{2}}-\tfrac{1}{2}{F_{3}}$, etc.

The odd coefficients do require one integration in their evaluation, and usually it is simpler to work in terms of $z$. In doing so let us denote $\hat{E}_{s}\left( z\right) =E_{s}\left( \xi \left( z\right) \right) $ and $\hat{{F}}_{s}\left( z\right) =F_{s}\left( \xi \left( z\right) \right) $, and it then follows from (\ref{eq2}) that%
\begin{equation}
\hat{E}_{s}\left( z\right) =\int \hat{{F}}_{s}\left( z\right) f^{1/2}\left(
z\right) dz\ \left( {s=1,2,3,\cdots }\right) ,  \label{eq12a}
\end{equation}%
where 
\begin{equation}
\hat{{F}}_{s+1}\left( z\right) =-{\dfrac{1}{2f^{1/2}\left( z\right) }}\frac{d\hat{{F}}_{s}\left( z\right) }{dz}-{\dfrac{1}{2}}\sum\limits_{j=1}^{s-1}{\hat{{F}}_{j}\left( z\right) \hat{{F}}_{s-j}\left( z\right) }\ \left( {s=2,3,4\cdots }\right) .  \label{eq13}
\end{equation}

In \cite{Dunster:1998} error bounds were derived for the expansions (\ref{eq6c}), but the ones given in that reference are harder to compute since they contain integrals involving the $E_{s}\left( \xi \right) $ coefficients. Here we shall provide simpler (as well as sharper) error bounds, and these will be used in a subsequent paper to provide readily computable error bounds for the expansions of turning point problems derived in \cite{DGS:2017}. Existing bounds for turning point expansions, although important theoretically, are quite complicated and very hard to compute (see \cite{Olver:1964} and \cite[chap. 11]{Olver:1997}).

We also plan to employ the L-G expansions of the form studied here, along with the new error bounds, to obtain simplified expansions and error bounds for more complicated situations, such as two coalescing turning points \cite{Olver:1975}, a coalescing turning point and simple pole \cite{Dunster:1994}, and a coalescing turning point and double pole \cite{BoydDunster:1986}.

To begin, we truncate the series in the formal expansions (\ref{eq6c}) after $n-1$ terms, and from these define solutions $W_{n,j}\left( {u,\xi }\right) $ ($j=1,2$) of (\ref{eq5}) in the form 
\begin{equation}
W_{n,1}\left( {u,\xi }\right) =\exp \left\{ {S_{n}\left( {u,\xi }\right) }\right\} \left\{ {e^{u\xi }+\varepsilon _{n,1}\left( {u,\xi }\right) } \right\} ,  \label{eq7}
\end{equation}%
and 
\begin{equation}
W_{n,2}\left( {u,\xi }\right) =\exp \left\{ {S_{n}\left( {-u,\xi }\right) } \right\} \left\{ {e^{-u\xi }+\varepsilon _{n,2}\left( {u,\xi }\right) } \right\} ,  \label{eq8}
\end{equation}%
where 
\begin{equation}
S_{n}\left( {u,\xi }\right) =\sum\limits_{s=1}^{n-1}{\frac{E_{s}\left( \xi
\right) }{u^{s}}.}  \label{eq9}
\end{equation}%

The focus of our attention are the error terms, and under certain conditions it can be shown that they satisfy $\varepsilon _{n,1}\left( {u,\xi }\right) =e^{u\xi }\mathcal{O}\left( {u^{-n}}\right) $ and $\varepsilon _{n,2}\left( {u,\xi }\right) =e^{-u\xi }\mathcal{O}\left( {u^{-n}}\right) $ as $u\rightarrow \infty $. Specifically, we shall
prove the following.

\begin{theorem}
For $j=1,2$ let the domain $\Xi _{j}\left( {u,\alpha _{j}}\right) $ comprise the $\xi $ point set for which there is a \textit{progressive} path $\mathcal{L}_{j}$ linking $\xi $ with $\alpha _{j}$ in $\Delta $, that is, one having the properties (i) $\mathcal{L}_{j}$ consists of a finite chain of $R_{2}$ arcs (as defined in \cite[chap. 5, sec. 3.3]{Olver:1997}), and (ii) as $t$ passes along $\mathcal{L}_{j}$ from $\alpha _{j}$ to $\xi $, the real part of $\left( {-1}\right) ^{j}ut$ is nonincreasing. Then for $\xi \in
\Xi _{j}\left( {u,\alpha _{j}}\right) $

\begin{multline}
\left\vert {\varepsilon _{n,1}\left( {u,\xi }\right) }\right\vert ,\
\left\vert \dfrac{1}{u}{\dfrac{\partial \varepsilon _{n,1}\left( {u,\xi }\right) }{\partial \xi }}\right\vert \leq \left\vert \dfrac{e^{u\xi }}{u^{n}}\right\vert \int_{\alpha _{1}}^{\xi }{\left\vert {\chi _{n}\left( {u,t}\right) dt}\right\vert } \\
\times \exp \left\{ {\dfrac{4}{\left\vert u\right\vert }\int_{\alpha
_{1}}^{\xi }{\left\vert {T_{n}\left( {u,t}\right) dt}\right\vert }+\dfrac{1}{\left\vert u\right\vert ^{n}}\int_{\alpha _{1}}^{\xi }{\left\vert {\chi
_{n}\left( {u,t}\right) dt}\right\vert }}\right\} ,  \label{eq23}
\end{multline}%
and

\begin{multline}
\left\vert {\varepsilon _{n,2}\left( {u,\xi }\right) }\right\vert ,\
\left\vert \dfrac{1}{u}{\dfrac{\partial \varepsilon _{n,2}\left( {u,\xi }\right) }{\partial \xi }}\right\vert \leq \left\vert \dfrac{e^{-u\xi }}{u^{n}}\right\vert \int_{\alpha _{2}}^{\xi }{\left\vert {\chi _{n}\left( -{u,t}\right) dt}\right\vert } \\
\times \exp \left\{ {\dfrac{4}{\left\vert u\right\vert }\int_{\alpha
_{2}}^{\xi }{\left\vert {T_{n}\left( -{u,t}\right) dt}\right\vert }+\dfrac{1}{\left\vert u\right\vert ^{n}}\int_{\alpha _{2}}^{\xi }{\left\vert {\chi
_{n}\left( -{u,t}\right) dt}\right\vert }}\right\} .  \label{eq26}
\end{multline}%
Here $T_{n}\left( {u,\xi }\right) $ and $\chi _{n}\left( {u,\xi }\right) $
are given by 
\begin{equation}
T_{n}\left( {u,\xi }\right) =u\frac{\partial S_{n}\left( {u,\xi }\right) }{\partial \xi }=\sum\limits_{s=0}^{n-2}{\frac{F_{s+1}\left( \xi \right) }{u^{s}},}  \label{eq16}
\end{equation}%
and

\begin{equation}
\chi _{n}\left( {u,\xi }\right) =2F_{n}\left( \xi \right)
-\sum\limits_{s=1}^{n-1}{\frac{G_{n,s}\left( \xi \right) }{u^{s}},}
\label{eq18}
\end{equation}%
where 
\begin{equation}
G_{n,s}\left( \xi \right) =\sum\limits_{k=s}^{n-1}{F_{k}\left( \xi \right)
F_{s+n-k-1}\left( \xi \right).}  \label{eq19}
\end{equation}
\label{thm1}
\end{theorem}

\textbf{Remark 1}. The domains can be unbounded provided the integrals in (\ref{eq23}) and (\ref{eq26}) converge at infinity. From \cite[chap. 10]{Olver:1997} a sufficient condition for this to be true is if there exist
the following convergent expansions in a neighbourhood of $z=\infty $%
\begin{equation}
f\left( z\right) =z^{m}\sum\limits_{s=0}^{\infty }{f}_{s}{z}^{-s},\ g\left(
z\right) =z^{p}\sum\limits_{s=0}^{\infty }{g}_{s}{z}^{-s},  \label{eq19a}
\end{equation}%
where the leading terms are nonzero, and $m>-2$ with $p<\frac{1}{2}m-1$, or $m=p=-2$ and $g_{0}=-\frac{1}{4}$. The same is true if $f\left( z\right) $ and $\ g\left( z\right) $ have a pole at $z=z_{0}$ (say), with Laurent
expansions in a deleted neighbourhood of $z=z_{0}$ of the form 
\begin{equation}
f\left( z\right) =\left( z-z_{0}\right) ^{-m}\sum\limits_{s=0}^{\infty }{f}_{s}\left( z-z_{0}\right) ^{s},\ g\left( z\right) =\left( z-z_{0}\right)^{-p}\sum\limits_{s=0}^{\infty }{g}_{s}\left( z-z_{0}\right) ^{s},
\label{eq19b}
\end{equation}%
where the leading terms are nonzero, and $m>2$ with $0\leq p<\frac{1}{2}m+1$, or $m=p=2$ and $g_{0}=-\frac{1}{4}$.

\textbf{Remark 2. }From (\ref{eq18}) we can utilise the simplification 
\begin{equation}
\int_{\mathcal{L}_{j}}{\left\vert {\chi _{n}\left( \pm {u,t}\right) dt} \right\vert }\leq 2\int_{\mathcal{L}_{j}}{\left\vert {F_{n}\left( t\right) dt}\right\vert }+\sum\limits_{s=1}^{n-1}{\frac{1}{\left\vert u\right\vert ^{s}}\int_{\mathcal{L}_{j}}{\left\vert {G_{n,s}\left( t\right) dt}\right\vert ,}}
\label{eq24}
\end{equation}%
and similarly from (\ref{eq16}) 
\begin{equation}
\int_{\mathcal{L}_{j}}{\left\vert {T_{n}\left( \pm {u,t}\right) dt}\right\vert }\leq \sum\limits_{s=0}^{n-2}{\frac{1}{\left\vert u\right\vert
^{s}}\int_{\mathcal{L}_{j}}{\left\vert {F_{s+1}\left( t\right) dt}\right\vert.}}  \label{eq25}
\end{equation}%
We also note that all functions appearing in these integrals can be explicitly evaluated without difficulty via (\ref{eq11}) and (\ref{eq12}), and hence these error bounds are readily computable via quadrature.

As we shall see, the integrals in the bounds can be quite large relative to the actual error when $\xi $ is not close to $\alpha _{j}.$ In order to obtain sharper bounds in this case we make the following simple modification, under the assumption that $\mathrm{Re}\left( u\alpha_{j}\right) =\left( -1\right)
^{j}\infty $ ($j=1,2$). Firstly, we note that the following solution is independent of $n$%
\begin{equation}
W_{1}\left( {u,\xi }\right) =\exp \left\{ {S_{n}\left( {u,\xi }\right)
-S_{n}\left( {u,}\alpha _{1}\right) }\right\} \left\{ {e^{u\xi }+\varepsilon
_{n,1}\left( {u,\xi }\right) }\right\} ,  \label{eq25a}
\end{equation}%
and indeed this is the unique solution having the property $W_{1}\left( {u,\xi }\right) ={e^{u\xi }}\left\{ 1+o\left( 1\right) \right\} $ as ${\xi
\rightarrow }\alpha _{1}$ ($\mathrm{Re}\left( u\xi \right) {\rightarrow
-\infty }$). Thus on replacing $n$ by $n+r$ for any positive integer $r$ we
deduce that the RHS of (\ref{eq25a}) is unchanged. As a result we can assert that 
\begin{multline}
\exp \left\{ {S_{n}\left( {u,\xi }\right) -S_{n}\left( {u,}\alpha
_{1}\right) }\right\} \left\{ {e^{u\xi }+\varepsilon _{n,1}\left( {u,\xi }\right) }\right\} \\
=\exp \left\{ {S_{n+r}\left( {u,\xi }\right) -S_{n+r}\left( {u,}\alpha
_{1}\right) }\right\} \left\{ {e^{u\xi }+\varepsilon _{n+r,1}\left( {u,\xi }\right) }\right\} ,  \label{eq25b}
\end{multline}%
and hence%
\begin{multline}
{\varepsilon _{n,1}\left( {u,\xi }\right) =e^{u\xi }}\left[ \exp \left\{ {S_{n+r}\left( {u,\xi }\right) -S_{n+r}\left( {u,}\alpha _{1}\right) -{S_{n}\left( {u,\xi }\right) }+S_{n}\left( {u,}\alpha _{1}\right) }\right\} -1\right] \\
+{\varepsilon _{n+r,1}\left( {u,\xi }\right) }\exp \left\{ {S_{n+r}\left( {u,\xi }\right) -S_{n+r}\left( {u,}\alpha _{1}\right) }-{S_{n}\left( {u,\xi }\right) +S_{n}\left( {u,}\alpha _{1}\right) }\right\} .  \label{eq25d}
\end{multline}

Likewise, if 
\begin{equation}
W_{2}\left( {u,\xi }\right) =\exp \left\{ {S_{n}\left( -{u,\xi }\right)
-S_{n}\left( -{u,}\alpha _{2}\right) }\right\} \left\{ {e^{-u\xi
}+\varepsilon _{n,2}\left( {u,\xi }\right) }\right\} ,  \label{eq25e}
\end{equation}%
then a similar expression can be derived for ${\varepsilon _{n,2}\left( {u,\xi }\right) }$ in terms of ${\varepsilon _{n+r,2}\left( {u,\xi }\right) }$. From theorem \ref{thm1} and (\ref{eq9}) we arrive at the following,
assuming that all integrals therein converge.

\begin{theorem}
Under the conditions of theorem \ref{thm1}, and with $\mathrm{Re}\left(u\alpha _{j}\right) =\left( -1\right) ^{j}\infty $ ($j=1,2$), the differential equation (\ref{eq5}) has solutions (\ref{eq25a}) and (\ref{eq25e}) which are independent of $n$, and whose error terms satisfy 
\begin{multline}
\left\vert {e^{-u\xi }\varepsilon _{n,1}\left( {u,\xi }\right) }\right\vert
\leq \left\vert \exp \left\{ \sum\limits_{s=n}^{n+r-1}{\dfrac{E_{s}\left(
\xi \right) -E_{s}\left( \alpha _{1}\right) }{u^{s}}}\right\} -1\right\vert
\\
+\dfrac{1}{\left\vert u\right\vert ^{n+r}}\int_{\alpha _{1}}^{\xi }{\left\vert {\chi _{n+r}\left( {u,t}\right) dt}\right\vert }\exp \left\{ {\dfrac{4}{\left\vert u\right\vert }\int_{\alpha _{1}}^{\xi }{\left\vert {T_{n+r}\left( {u,t}\right) dt}\right\vert }}\right. \\
\left._{\overset{}{}}^{{}}{+\dfrac{1}{\left\vert u\right\vert ^{n+r}}\int_{\alpha _{1}}^{\xi }{\left\vert {\chi _{n+r}\left( {u,t}\right) dt}\right\vert +}}\mathrm{Re}\sum\limits_{s=n}^{n+r-1}{\dfrac{E_{s}\left( \xi
\right) -E_{s}\left( \alpha _{1}\right) }{u^{s}}}\right\},  \label{eq26a}
\end{multline}
and 
\begin{multline}
\left\vert {e^{u\xi }\varepsilon _{n,2}\left( {u,\xi }\right) }\right\vert
\leq \left\vert \exp \left\{ \sum\limits_{s=n}^{n+r-1}\left( -1\right) ^{s}{\dfrac{E_{s}\left( \xi \right) -E_{s}\left( \alpha _{2}\right) }{u^{s}}}\right\} -1\right\vert \\
+\dfrac{1}{\left\vert u\right\vert ^{n+r}}\int_{\alpha _{2}}^{\xi }{\left\vert {\chi _{n+r}\left( -{u,t}\right) dt}\right\vert }\exp \left\{ {\dfrac{4}{\left\vert u\right\vert }\int_{\alpha _{2}}^{\xi }{\left\vert {T_{n+r}\left( -{u,t}\right) dt}\right\vert }}\right. \\
\left._{\overset{}{}}^{{}}{+\dfrac{1}{\left\vert u\right\vert ^{n+r}}\int_{\alpha _{2}}^{\xi }{\left\vert {\chi _{n+r}\left( -{u,t}\right) dt}\right\vert +}}\mathrm{Re}\sum\limits_{s=n}^{n+r-1}\left( -1\right) ^{s}{\dfrac{E_{s}\left( \xi \right) -E_{s}\left( \alpha _{2}\right) }{u^{s}}}\right\},  \label{eq26b}
\end{multline}%
for any positive integer $r$. \label{thm2}
\end{theorem}

\textbf{Remark}: In computing the first terms on the RHS of both bounds for
large $\left\vert u\right\vert $ it may be numerically advantageous to use
the bound%
\begin{equation}
\left\vert e^{z}-1\right\vert \leq \sum\limits_{k=1}^{\infty }{\dfrac{\left\vert z\right\vert ^{k}}{k!}.}  \label{eq26c}
\end{equation}

The plan of this paper is as follows. In section \ref{sec2} we prove theorem \ref{thm1}, and extend this to the case where the error terms appear in the exponent of the exponential approximations. In section \ref{sec3} generalisations are given of theorem \ref{thm1} for certain cases where $f=f\left( u,z\right) $ is dependent on $u$. A nonhomogeneous form of (\ref{eq1}) is studied in section \ref{sec4}, and an explicit error bound is given for an asymptotic expansion of a uniquely-defined particular solution. Finally, in section \ref{sec5}, theorem \ref{thm2} is applied to the modified Bessel equation to yield exponential-form uniform asymptotic expansions for its solutions, complete with error bounds which are sharp and easily computable.

\section{Proof and extensions of theorem 1.1}

\label{sec2}

Firstly we consider $\varepsilon =\varepsilon _{n,1}\left( {u,\xi }\right) $. Then on denoting $T=T_{n}\left( {u,\xi }\right) _{\,}$ we have from (\ref{eq5}), (\ref{eq10})-(\ref{eq12}), (\ref{eq7}) and (\ref{eq16}) 
\begin{equation}
d^{2}\varepsilon /d\xi ^{2}-u^{2}\varepsilon =h,  \label{eq14}
\end{equation}%
where 
\begin{equation}
h=h_{n}\left( {u,\xi }\right) =2T\varepsilon +\frac{\chi \varepsilon }{u^{n-1}}-\frac{2T}{u}\frac{d\varepsilon }{d\xi }+\frac{e^{u\xi }\chi }{u^{n-1}},  \label{eq15}
\end{equation}%
in which 
\begin{equation}
\chi =\chi_{n}\left( {u,\xi }\right) =u^{n-1}\left\{ {\psi -2T-\frac{1}{u}\frac{dT}{d\xi }-\frac{T^{2}}{u^{2}}}\right\} .  \label{eq17}
\end{equation}%
Now using the Cauchy product formula, along with (\ref{eq12}) and (\ref{eq16}), we arrive at (\ref{eq18}), and in particular $\chi_{n}\left( {u,\xi }\right)$ is bounded in $\Delta$ as $u\rightarrow \infty $.

Next, from (\ref{eq14}) we apply variation of parameters to obtain the
integral equation 
\begin{equation}
\varepsilon _{n,1}\left( {u,\xi }\right) =\frac{1}{u}\int_{\alpha _{1}}^{\xi
}\mathsf{K}{\left( {\xi ,t}\right) h}_{n}{\left( {u,t}\right) dt,}
\label{eq20}
\end{equation}%
where 
\begin{equation}
\mathsf{K}\left( {\xi ,t}\right) ={\tfrac{1}{2}}\left\{ {e^{u\left( {\xi -t}\right) }-e^{u\left( {t-\xi }\right) }}\right\},  \label{eq21}
\end{equation}%
and the path of integration is taken to be $\mathcal{L}_{1}$. Therefore,
under the conditions of this progressive path we deduce that 
\begin{equation}
\left\vert \mathsf{K}\left( {\xi ,t}\right) \right\vert \leq P_{0}\left( {\xi }\right) Q\left( t\right) ,\ \left\vert \partial \mathsf{K}\left( {\xi ,t}\right) /\partial {\xi }\right\vert \leq P_{1}\left( {\xi }\right) Q\left(
t\right) ,  \label{eq57a}
\end{equation}%
where%
\begin{equation}
P_{0}\left( {\xi }\right) =\left\vert {e^{u\xi }}\right\vert ,\ P_{1}\left( {\xi }\right) =\left\vert {u e^{u\xi }}\right\vert ,\ Q\left( {t}\right)=\left\vert {e^{-ut}}\right\vert .  \label{eq21a}
\end{equation}

We now apply \cite[chap. 6, theorem 10.1]{Olver:1997} to (\ref{eq20}). In
Olver's notation we have $P_{0}$,$\ P_{1}$,$\ $and $Q$ as above, and on
referring to (\ref{eq15}) 
\begin{equation}
J=e^{ut},\ \phi =\frac{\chi }{u^{n}},\ \psi _{0}=\frac{2T}{u}+\frac{\chi }{u^{n}},\ \psi _{1}=-\frac{2T}{u^{2}}.  \label{eq22}
\end{equation}%
This then establishes (\ref{eq23}). The bound (\ref{eq26}) is proved
similarly.

In studying zeros it is sometimes convenient to have the error term in the
exponent. From (\ref{eq7}) 
\begin{equation}
W_{n,1}\left( {u,\xi }\right) =\exp \left\{ {u\xi +S_{n}\left( {u,\xi }\right) +\delta _{n,1}\left( {u,\xi }\right) }\right\},  
\label{eq27}
\end{equation}%
and%
\begin{equation}
W_{n,2}\left( {u,\xi }\right) =\exp \left\{ -{u\xi +S_{n}\left( -{u,\xi }\right) +\delta _{n,2}\left( {u,\xi }\right) }\right\},
\label{eq27a}
\end{equation}%
where for $j=1,2$ 
\begin{equation}
\delta _{n,j}\left( {u,\xi }\right) =\ln \left[ {1+\exp }\left\{ \left(
-1\right) ^{j+1}u\xi \right\} {\varepsilon _{n,j}\left( {u,\xi }\right) }\right].  \label{eq28}
\end{equation}%
Hence from (\ref{eq23}) and (\ref{eq26}) we immediately infer that 
\begin{equation}
\left\vert {\delta _{n,j}\left( {u,\xi }\right) }\right\vert \leq -\ln
\left\{ {1-\kappa _{n,j}\left( {u,\xi }\right) }\right\} ,  \label{eq29}
\end{equation}%
where 
\begin{multline}
\kappa _{n,j}\left( {u,\xi }\right) =\frac{1}{\left\vert u\right\vert ^{n}}\int_{\alpha _{j}}^{\xi }{\left\vert {\chi _{n}\left( \pm {u,t}\right) dt}\right\vert } \\
\times \exp \left\{ {\frac{4}{\left\vert u\right\vert }\int_{\alpha
_{j}}^{\xi }{\left\vert {T_{n}\left( \pm {u,t}\right) dt}\right\vert }+\frac{1}{\left\vert u\right\vert ^{n}}\int_{\alpha _{j}}^{\xi }{\left\vert {\chi
_{n}\left( \pm {u,t}\right) dt}\right\vert }}\right\} ,  \label{eq30}
\end{multline}%
provided $\left\vert u\right\vert $ is sufficiently large so that $\kappa
_{n,j}\left( {u,\xi }\right) <1$. Here and throughout the plus signs are
taken for $j=1$ and the minus signs for $j=2$. We note that (\ref{eq24}) and
(\ref{eq25}) can again be employed to simplify computation.

Next for the derivatives of solutions we obtain from (\ref{eq16}), (\ref{eq27}) and (\ref{eq28}) 
\begin{multline}
\frac{\partial W_{n,1}\left( {u,\xi }\right) }{\partial \xi }=\left[ {u+\frac{T_{n}\left( {u,\xi }\right) }{u}+\frac{\partial \delta _{n,1}\left( {u,\xi }\right) }{\partial \xi }}\right] \\
\times \exp \left\{ {u\xi +S_{n}\left( {u,\xi }\right) +\delta _{n,1}\left( {u,\xi }\right) }\right\} .  \label{eq31}
\end{multline}%
Consequently 
\begin{multline}
\exp \left\{ {-u\xi -S_{n}\left( {u,\xi }\right) }\right\} \dfrac{\partial
W_{n,1}\left( {u,\xi }\right) }{\partial \xi } \\
=u+\dfrac{T_{n}\left( {u,\xi }\right) }{u}\left\{ {1+e^{-u\xi }\varepsilon
_{n,1}\left( {u,\xi }\right) }\right\} +e^{-u\xi }\dfrac{\partial
\varepsilon _{n,1}\left( {u,\xi }\right) }{\partial \xi }.  \label{eq32}
\end{multline}%
Accordingly 
\begin{equation}
\frac{\partial W_{n,1}\left( {u,\xi }\right) }{\partial \xi }=u\exp \left\{ {u\xi +S_{n}\left( {u,\xi }\right) +\ln \left\{ {1+u^{-2}T_{n}\left( {u,\xi }\right) }\right\} +\eta _{n,1}\left( {u,\xi }\right) }\right\} ,
\label{eq33}
\end{equation}%
where 
\begin{multline}
\exp \left\{ {\eta _{n,1}\left( {u,\xi }\right) }\right\} =1+{e^{-u\xi }}\left[ {1+\frac{T_{n}\left( {u,\xi }\right) }{u^{2}}}\right] ^{-1} \\
\times \left\{ {\frac{\partial \varepsilon _{n,1}\left( {u,\xi }\right) }{u\partial \xi }+\frac{T_{n}\left( {u,\xi }\right) }{u^{2}}\varepsilon
_{n,1}\left( {u,\xi }\right) }\right\} .  \label{eq34}
\end{multline}%
This leads to the desired bound 
\begin{equation}
\left\vert {\eta _{n,1}\left( {u,\xi }\right) }\right\vert \leq -\ln \left[ {1-\frac{\left\{ \left\vert {u}\right\vert {^{2}+\left\vert {T_{n}\left( {u,\xi }\right) }\right\vert }\right\} \kappa _{n,1}\left( {u,\xi }\right) }{\left\vert {u}\right\vert ^{2}-\left\vert {T_{n}\left( {u,\xi }\right) }\right\vert }}\right] ,  \label{eq35}
\end{equation}%
in which $\kappa _{n,1}\left( {u,\xi }\right) $ is given by (\ref{eq30}).

Similarly we find
\begin{multline}
\dfrac{\partial W_{n,2}\left( {u,\xi }\right) }{\partial \xi }=-u\exp
\left\{ -{u\xi +S_{n}\left( -{u,\xi }\right) }\right. \\
\left. {+\ln }\left[ {1+u^{-2}T_{n}\left( -{u,\xi }\right) }\right] {+\eta
_{n,2}\left( {u,\xi }\right) }\right\} ,  \label{eq35a}
\end{multline}%
where%
\begin{equation}
\left\vert {\eta _{n,2}\left( {u,\xi }\right) }\right\vert \leq -\ln \left[ {1-\frac{\left\{ \left\vert {u}\right\vert {^{2}+\left\vert {T_{n}\left( -{u,\xi }\right) }\right\vert }\right\} \kappa _{n,2}\left( {u,\xi }\right) }{\left\vert {u}\right\vert ^{2}-\left\vert {T_{n}\left( -{u,\xi }\right) }\right\vert }}\right] .  \label{eq35b}
\end{equation}

More generally, for any non-vanishing differentiable function $\sigma \left(
\xi \right) $, one can similarly demonstrate that

\begin{multline}
\dfrac{\partial \left\{ {\sigma \left( \xi \right) W_{n,j}\left( {u,\xi }\right) }\right\} }{\partial \xi }=\pm u\sigma \left( \xi \right) \exp
\left\{ \pm {u\xi +\ln \left[ {1\pm \dfrac{{\sigma }^{\prime }\left( \xi
\right) }{u\sigma \left( \xi \right) }+\dfrac{T_{n}\left( \pm {u,\xi }\right) }{u^{2}}}\right] }\right. \\
\biggl._{\overset{}{}}^{{}}+{S_{n}\left( \pm {u,\xi }\right) +\eta
_{n,j}\left( {u,\xi }\right) }\biggr\},  \label{eq36}
\end{multline}%
where 
\begin{equation}
\left\vert {\eta _{n,j}\left( {u,\xi }\right) }\right\vert \leq -\ln \left[ {1-\frac{\left\{ {\left\vert {u^{2}\sigma \left( \xi \right) }\right\vert
+\left\vert \pm {u{\sigma }^{\prime }\left( \xi \right) +\sigma \left( \xi
\right) T_{n}\left( \pm {u,\xi }\right) }\right\vert }\right\} \kappa
_{n,j}\left( {u,\xi }\right) }{\left\vert u^{2}{\sigma \left( \xi \right) }\right\vert -\left\vert \pm {u{\sigma }^{\prime }\left( \xi \right) +\sigma
\left( \xi \right) T_{n}\left( \pm {u,\xi }\right) }\right\vert }}\right] .
\label{eq37}
\end{equation}

\section{More general form of equation}

\label{sec3}

In this section we consider the case where $f=f\left( u,z\right) $ is
dependent on $u$. Firstly, suppose that $f\left( u,z\right) ={f_{0}\left(
z\right) +u}^{-1}{f_{1}\left( z\right) }$, where ${f_{0}\left( z\right) }$
and ${f_{1}\left( z\right) }$ are independent of $u$. Then (\ref{eq1}) takes
the form 
\begin{equation}
d^{2}w/dz^{2}=\left\{ {u^{2}f_{0}\left( z\right) +uf_{1}\left( z\right)
+g\left( z\right) }\right\} w.  \label{eq38}
\end{equation}

In place of (\ref{eq3}) and (\ref{eq4}) the appropriate Liouville transformation now reads as 
\begin{equation}
\xi =\int {f_{0}^{1/2}\left( z\right) dz},\ W=f_{0}^{1/4}\left( z\right) w.
\label{eq39}
\end{equation}%
As a result, we obtain the transformed equation 
\begin{equation}
d^{2}W/d\xi ^{2}=\left\{ {u^{2}+u\phi \left( \xi \right) +\psi \left( \xi
\right) }\right\} W,  \label{eq40}
\end{equation}%
where 
\begin{equation}
\phi \left( \xi \right) =f_{1}\left( z\right) /f_{0}\left( z\right) ,
\label{eq41}
\end{equation}%
and 
\begin{equation}
\psi \left( \xi \right) =\frac{4f_{0}\left( z\right) {f}_{0}^{\prime \prime
}\left( z\right) -5{f}_{0}^{\prime }{}^{2}\left( z\right) }{16f_{0}^{3}\left( z\right) }+\frac{g\left( z\right) }{f_{0}\left( z\right) }.
\label{eq42}
\end{equation}

In place of (\ref{eq10}) - (\ref{eq12}) the coefficients in the expansions
are found to be given by 
\begin{equation}
E_{s}^{\pm }\left( \xi \right) =\int {F_{s}^{\pm }\left( \xi \right) d\xi }\
\left( {s=0,1,2,\cdots }\right) ,  \label{eq46}
\end{equation}%
where 
\begin{equation}
{F_{0}^{\pm }}\left( \xi \right) =\pm {\tfrac{1}{2}}\phi \left( \xi \right) ,
\label{eq47}
\end{equation}%
\begin{equation}
{F_{1}^{\pm }}\left( \xi \right) ={\tfrac{1}{2}}\psi \left( \xi \right) -{\tfrac{1}{8}}\phi ^{2}\left( \xi \right) \mp {\tfrac{1}{4}}{\phi }^{\prime
}\left( \xi \right) ,  \label{eq48}
\end{equation}%
and 
\begin{equation}
{F_{s+1}^{\pm }}\left( \xi \right) =-{\dfrac{1}{2}}\frac{d{F_{s}^{\pm }}\left( \xi \right) }{d\xi }-{\dfrac{1}{2}}\sum\limits_{j=0}^{s}{F_{j}^{\pm
}\left( \xi \right) F_{s-j}^{\pm }\left( \xi \right) }\ \left( {s=1,2,3,\cdots }\right) .  \label{eq49}
\end{equation}%

We then define $\chi ^{\pm }=\chi _{n}^{\pm }\left( {u,\xi }\right) $ by (c.f. (\ref{eq17}))%
\begin{equation}
\chi ^{\pm }=\left( \pm u\right) ^{n-1}\left[ \psi \mp \frac{1}{2}{\phi }^{\prime }-\frac{1}{4}\phi ^{2}-2T^{\pm }-\frac{\phi T^{\pm }}{u}\mp \frac{1}{u}\frac{d T^{\pm }}{d \xi }-\frac{\left\{ T^{\pm }\right\} ^{2}}{u^{2}}\right] .  \label{eq51a}
\end{equation}%
Here $T=T_{n}^{\pm }\left( {u,\xi }\right) $ is defined by%
\begin{equation}
T_{n}^{\pm }\left( {u,\xi }\right) =\sum\limits_{s=0}^{n-2}\left( \pm
1\right) ^{s}{\frac{F_{s+1}^{\pm }\left( \xi \right) }{u^{s}},}
\label{eq51b}
\end{equation}
and use of the Cauchy product formula establishes that $\chi _{n}^{\pm }\left( {u,\xi }\right) $ are both bounded in $\Delta $ as $u \rightarrow \infty $.

With the definition%
\begin{equation}
S_{n}^{\pm }\left( {u,\xi }\right) =\sum\limits_{s=1}^{n-1}\left( \pm
1\right) ^{s}{\frac{E_{s+1}^{\pm }\left( \xi \right) }{u^{s}},}
\label{eq51c}
\end{equation}%
we then arrive at the following.

\begin{theorem}
For $j=1,2$ let the domain $\Xi _{j}\left( {u,\alpha _{j}}\right) $ comprise
the $\xi $ point set for which there is a \textit{path} $\mathcal{L}_{j}$
linking $\xi $ with $\alpha _{j}$ in $\Delta $ having the properties (i) $\mathcal{L}_{j}$ consists of a finite chain of $R_{2}$ arcs, and (ii) as $t$
passes along $\mathcal{L}_{j}$ from $\alpha _{j}$ to $\xi $, the real part
of $\left( {-1}\right) ^{j}ut{-E_{0}^{\pm}\left( t\right) }$ is nonincreasing. Then for $\xi \in \Xi _{j}\left( {u,\alpha _{j}}\right) $ the
differential equation (\ref{eq40}) has solutions of the form 
\begin{equation}
W_{n,1}\left( {u,\xi }\right) =\exp \left\{ S_{n}^{+}\left( {u,\xi }\right)
\right\} \left[ \exp \left\{ u\xi +{E_{0}^{+}\left( \xi \right) }\right\} {+\varepsilon _{n,1}\left( {u,\xi }\right) }\right] ,  \label{eq43}
\end{equation}%
and 
\begin{equation}
W_{n,2}\left( {u,\xi }\right) =\exp \left\{ S_{n}^{-}\left( {u,\xi }\right)
\right\} \left[ \exp \left\{ -u\xi +{E_{0}^{-}\left( \xi \right) }\right\} {+\varepsilon _{n,2}\left( {u,\xi }\right) }\right] ,  \label{eq44}
\end{equation}%
where 
\begin{multline}
\left\vert {\varepsilon _{n,j}\left( {u,\xi }\right) }\right\vert ,\
\left\vert {\dfrac{1}{u+{\frac{1}{2}}\phi \left( \xi \right) }\dfrac{\partial \varepsilon _{n,j}\left( {u,\xi }\right) }{\partial \xi }}\right\vert \leq \kappa _{0}\left( u,\xi \right) \left\vert \dfrac{\exp
\left\{ \pm u\xi {+E_{0}^{\pm }\left( \xi \right) }\right\} }{u^{n}}%
\right\vert \\
\times \int_{\mathcal{L}_{j}}{\left\vert \chi _{n}^{\pm }\left( {u,t}\right) 
{dt}\right\vert }\exp \left\{ {\dfrac{1}{\left\vert u\right\vert }\left( {2+2\kappa _{0}\left( u,\xi \right) +\dfrac{\kappa _{0}\left( u,\xi \right)
\kappa _{2}{\left( \xi \right) }}{\left\vert u\right\vert }}\right) \int_{\mathcal{L}_{j}}{\left\vert {T_{n}^{\pm }}\left( {u,t}\right) {dt}\right\vert }}\right. \\
\left. {+\dfrac{\kappa _{0}\left( u,\xi \right) }{\left\vert u\right\vert }\int_{\mathcal{L}_{j}}{\left\vert {{\phi }^{\prime }\left( t\right) dt}\right\vert }+\dfrac{\kappa _{0}\left( u,\xi \right) }{\left\vert
u\right\vert ^{n}}\int_{\mathcal{L}_{j}}{\left\vert \chi _{n}^{\pm }\left( {u,t}\right) {dt}\right\vert }}\right\} .  \label{eq59}
\end{multline}%
Here 
\begin{equation}
\kappa _{0}\left( u,\xi \right) =\sup_{t\in \mathcal{L}_{j}}\left\{ \frac{1}{\left\vert {1+\phi \left( t\right) /\left( {2u}\right) }\right\vert }\right\} ,\ \kappa _{2}\left( \xi \right) =\sup_{t\in \mathcal{L}_{j}}\left\vert {\phi \left( t\right) }\right\vert .
\end{equation}%
\label{thm3}
\end{theorem}

\begin{proof}
Again we consider $\varepsilon =\varepsilon _{n,1}\left( {u,\xi }\right) $,
with $\varepsilon _{n,2}\left( {u,\xi }\right) $ proved similarly. From (\ref{eq40}), (\ref{eq43}), (\ref{eq46}) - (\ref{eq49}) we see that the error
term in question satisfies the differential equation 
\begin{equation}
d^{2}\varepsilon /d\xi ^{2}-\left\{ {u^{2}+u\phi \ }\right\} \varepsilon =h,  \label{eq50}
\end{equation}%
where this time $h$ is given by 
\begin{equation}
h=2T^{+}\varepsilon +\frac{1}{2}{\phi }^{\prime }\varepsilon +\frac{\phi
T^{+}\varepsilon }{u} +\frac{\chi ^{+}\varepsilon }{u^{n-1}}+\frac{1}{4}\phi
^{2}\varepsilon -\frac{2T^{+}}{u}\frac{d\varepsilon }{d\xi }+\frac{\exp
\left( u\xi +E_{0}^{+}\right) \chi ^{+}}{u^{n-1}}.  \label{eq51}
\end{equation}%

Now it is verifiable by substitution that any linearly independent pair of
twice differentiable functions $y_{1}\left( \xi \right) $ and $y_{2}\left(\xi \right) $ satisfy the linear second order differential equation 
\begin{equation}
\frac{d}{d\xi }\left\{ {\frac{y}{y_{1}}-\frac{y_{2}}{y_{1}\left( {y_{2}/y_{1}}\right) ^{\prime }}\frac{d}{d\xi }\left( {\frac{y}{y_{1}}}\right) }\right\}
=0,  \label{eq52}
\end{equation}%
and so on using this, along with (\ref{eq46}) and (\ref{eq47}), we find that 
$y=\exp \left\{ {\pm u\xi +E_{0}^{\pm }\left( \xi \right) }\right\} $ are
solutions of 
\begin{equation}
\frac{d^{2}y}{d\xi ^{2}}-\frac{{\phi }^{\prime }}{2u+\phi }\frac{dy}{d\xi }-\left\{ {u^{2}+u\phi +\frac{1}{4}\phi ^{2}}\right\} y=0.  \label{eq53}
\end{equation}%
We thus reformulate (\ref{eq50}) into the form 
\begin{equation}
\frac{d^{2}\varepsilon }{d\xi ^{2}}-\frac{{\phi }^{\prime }}{2u+\phi }\frac{d\varepsilon }{d\xi }-\left\{ {u^{2}+u\phi +\frac{1}{4}\phi ^{2}}\right\}
\varepsilon =h^{+},  \label{eq54}
\end{equation}%
where, from (\ref{eq51}), 
\begin{multline}
h^{+}=h_{n,1}^{+}\left( {u,\xi }\right) =2T^{+}\varepsilon +\frac{1}{2}{\phi 
}^{\prime }\varepsilon +\frac{\phi T^{+}\varepsilon }{u}+\frac{\chi^{+}\varepsilon }{u^{n-1}}-\frac{2T^{+}}{u}\frac{d\varepsilon }{d\xi } \\
-\frac{{\phi }^{\prime }}{2u+\phi }\frac{d\varepsilon }{d\xi }+\frac{\exp\left( u\xi +E_{0}^{+}\right) \chi ^{+}}{u^{n-1}}.  \label{eq55}
\end{multline}%

Now using the Wronskian 
\begin{equation}
\mathcal{W}\left[ {\exp \left\{ {u\xi +E_{0}^{+}\left( \xi \right) }\right\}
,\exp \left\{ {-u\xi +E_{0}^{-}\left( \xi \right) }\right\} }\right]
=-\left\{ {2u+\phi \left( \xi \right) }\right\} ^{-1},  \label{eq56}
\end{equation}%
along with variation parameters, we find from (\ref{eq54}) that
\begin{equation}
\varepsilon _{n,1}\left( {u,\xi }\right) =\frac{1}{u}\int_{\alpha _{1}}^{\xi
}{\mathsf{K}\left( {\xi ,t}\right) h_{n,1}^{+}\left( {u,t}\right) dt,}
\label{eq57}
\end{equation}
where 
\begin{equation}
\mathsf{K}\left( {\xi ,t}\right) =\frac{\exp \left\{ {u\xi +E_{0}^{+}\left(
\xi \right) -ut-E_{0}^{+}\left( t\right) }\right\} -\exp \left\{ {ut+E_{0}^{+}\left( t\right) -u\xi -E_{0}^{+}\left( \xi \right) }\right\} }{2+u^{-1}\phi \left( t\right) }.  \label{eq58}
\end{equation}

Next, by taking the path of integration in (\ref{eq57}) as the progressive path $\mathcal{L}_{1}$ we
deduce that (\ref{eq57a}) holds, where this time 
\begin{multline}
P_{0}\left( {\xi }\right) =\left\vert \exp \left\{ {u\xi +E_{0}^{+}\left(
\xi \right) }\right\} \right\vert ,\ Q\left( t\right) =\left\vert \frac{\exp
\left\{ {-ut-E_{0}^{+}\left( t\right) }\right\} }{1+\frac{1}{2}u^{-1}\phi
\left( t\right) }\right\vert , \\
\ P_{1}\left( {\xi }\right) =\left\vert \left( u+\tfrac{1}{2}\phi \left( \xi
\right) \right) \exp \left\{ {u\xi +E_{0}^{+}\left( \xi \right) }\right\}
\right\vert .  \label{eq58a}
\end{multline}%
We then apply \cite[chap. 6, theorem 10.1]{Olver:1997} with%
\begin{multline}
J=\exp \left\{ ut+E_{0}^{+}\left( t\right) \right\} ,\ \phi ^{olv}=\frac{\chi ^{+}}{u^{n}},\ \psi _{0}=\frac{2T^{+}}{u}+\frac{{\phi }^{\prime }}{2u}+\frac{\chi ^{+}}{u^{n}}+\frac{\phi T^{+}}{u^{2}},\  \\
\psi _{1}=-\frac{2T^{+}}{u^{2}}-\frac{{\phi }^{\prime }}{u\left( 2u+\phi
\right) }.  \label{eq58b}
\end{multline}%
Here we have used the notation $\phi ^{olv}$ for Olver's $\phi $ to avoid
conflict with ours.
\end{proof}

For a more general case let us assume that%
\begin{equation}
f\left( u,z\right) \sim \sum\limits_{s=0}^{\infty }\frac{f{_{s}\left(z\right) }}{u^{s}}.
\label{eq601}
\end{equation}%
We again use the transformation (\ref{eq39}), and as a result (\ref{eq1}) becomes 
\begin{equation}
d^{2}W/d\xi ^{2}=\left\{ {u^{2}+u\phi \left( \xi \right) +\psi \left( u,\xi
\right) }\right\} W,  \label{eq601a}
\end{equation}%
where%
\begin{equation}
\psi \left( u,\xi \right) =\sum\limits_{s=0}^{n-1}\frac{\psi {_{s}\left( \xi
\right) }}{u^{s}}{+}\frac{{\Psi }_{n}\left( u,\xi \right) }{u^{n}},
\label{eq602}
\end{equation}%
with $\psi {_{0}\left( \xi \right) }$ again given by (\ref{eq42}), subsequent terms are furnished by 
\begin{equation}
\psi {_{s}\left( \xi \right) =}\frac{f{_{s+2}\left( z\right) }}{f{_{0}\left(z\right) }}\ \left( {s=1,2,3,\cdots }\right) ,  \label{eq603}
\end{equation}%
and ${\Psi }_{n}\left( u,\xi \right) $ is analytic and bounded in $\Delta $.

We find that $F_{0}^{\pm }\left( \xi \right) $ are again given by (\ref{eq47}), with subsequent terms modified by 
\begin{equation}
F_{1}^{\pm }\left( \xi \right) ={\tfrac{1}{2}}\psi _{0}\left( \xi \right) -{\tfrac{1}{8}}\phi ^{2}\left( \xi \right) \mp {\tfrac{1}{4}}{\phi }^{\prime
}\left( \xi \right) ,  \label{eq604}
\end{equation}%
and%
\begin{equation}
F_{s+1}^{\pm }\left( \xi \right) ={\dfrac{1}{2}}\psi _{s}\left( \xi \right) -{\dfrac{1}{2}}\frac{dF_{s}^{\pm }\left( \xi \right) }{d\xi }-{\dfrac{1}{2}}\sum\limits_{j=0}^{s}{F_{j}^{\pm }\left( \xi \right) F_{s-j}^{\pm }\left(
\xi \right) }\ \left( {s=1,2,3,\cdots }\right) .  \label{eq605}
\end{equation}%
As a result solutions of the forms (\ref{eq43}) and (\ref{eq44}) hold, with coefficients modified as above (and with (\ref{eq46}) still applicable). In these expansions, the error bounds are given by (\ref{eq59}) in the same domains, but where $\psi \left( \xi \right) $ in (\ref{eq51a})\ is replaced
by the expansion (\ref{eq602}).

\section{Nonhomogeneous equations}

\label{sec4}

In this section we consider the following nonhomogeneous form of (\ref{eq1}) 
\begin{equation}
d^{2}w/dz^{2}-\left\{ {u^{2}f\left( z\right) +g\left( z\right) }\right\}
w=p\left( z\right) .  \label{eq61}
\end{equation}%
Here $p\left( z\right) $ is analytic in $\mathbf{D}$, and ${f\left( z\right) 
}$ (nonvanishing), ${g\left( z\right) }$, and $p\left( z\right) $ are
independent of $u$ (although the latter restriction can often be relaxed without undue difficulty). The Liouville transformation (\ref{eq3}) and (\ref{eq4}) is again applicable, and this results in the transformed
nonhomogeneous equation 
\begin{equation}
d^{2}W/d\xi ^{2}-\left\{ {u^{2}+\psi \left( \xi \right) }\right\} W=\varpi
\left( \xi \right) ,  \label{eq62}
\end{equation}%
where 
\begin{equation}
\varpi \left( \xi \right) =f^{-3/4}\left( z\right) p\left( z\right) ,
\label{eq63}
\end{equation}%
which is analytic in $\Delta $.

Neglecting $\varpi \left( \xi \right) $ we have the homogeneous equation (\ref{eq5}), with asymptotic solutions given by (\ref{eq7}) and (\ref{eq8}).
A general solution of (\ref{eq62}) is therefore furnished by 
\begin{equation}
W\left( {u,\xi }\right) =A\left( u\right) W_{n,1}\left( {u,\xi }\right)
+B\left( u\right) W_{n,2}\left( {u,\xi }\right) +G\left( {u,\xi }\right) ,
\label{eq66}
\end{equation}%
for arbitrary constants $A\left( u\right) $ and $B\left( u\right)$, and  $G\left( {u,\xi }\right) $ is any particular solution. The focus of our
attention is the asymptotic approximation of this latter solution.

Following \cite[chap. 10, section 10]{Olver:1997} if we assume the expansion 
\begin{equation}
G\left( {u,\xi }\right) \sim \frac{1}{u^{2}}\sum\limits_{s=0}^{\infty }{\frac{G_{s}\left( \xi \right) }{u^{2s}},}  \label{eq66a}
\end{equation}%
we find on substitution into (\ref{eq62}) and equating powers of $u$, that $G_{0}\left( \xi \right) =-\varpi \left( \xi \right) $ and 
\begin{equation}
G_{s+1}\left( \xi \right) ={G}_{s}^{\prime \prime }\left( \xi \right) -\psi
\left( \xi \right) G_{s}\left( \xi \right) \ \left( {s=0,1,2,\cdots }\right)
.  \label{eq68}
\end{equation}

An explicit error bound for the expansion (\ref{eq66a}) is furnished as
follows. In this we assume that $\mathrm{Re}\left( u\alpha _{j}\right)
=\left( -1\right) ^{j}\infty $, although this is not a critical assumption.

\begin{theorem}
Let $\mathcal{L}\left( \xi \right) $ be a path in $\Xi _{1}\left( {u,\alpha
_{1}}\right) \cap \Xi _{2}\left( {u,\alpha _{2}}\right) $ linking $\alpha_{1}$ to $\alpha _{2}$ that contains the point ${\xi }$, having the
properties (i) $\mathcal{L}\left( \xi \right) $ consists of a finite chain
of $R_{2}$ arcs, and (ii) as $t$ passes along $\mathcal{L}\left( \xi \right) 
$ from $\alpha _{1}$ to $\alpha _{2}$, the real part of $ut$ is monotonic.
Further assume $\left\vert u\right\vert $ is sufficiently large so that $\int_{\mathcal{L}\left( \xi \right) }{\left\vert {\psi \left( t\right) dt}\right\vert }<2\left\vert u\right\vert $. Then there exists a unique particular solution of (\ref{eq62}) of the form 
\begin{equation}
G\left( {u,\xi }\right) =\frac{1}{u^{2}}\sum\limits_{s=0}^{n-1}{\frac{G_{s}\left( \xi \right) }{u^{2s}}}+\varepsilon _{n}\left( {u,\xi }\right) ,
\label{eq67}
\end{equation}%
where 
\begin{multline}
\left\vert {\varepsilon _{n}\left( {u,\xi }\right) }\right\vert \leq \dfrac{1}{\left\vert u\right\vert ^{2n+2}}\left\{ \left\vert {G_{n}\left( \xi
\right) }\right\vert +\dfrac{1}{2}\int_{\mathcal{L}\left( \xi \right) }{\left\vert {{G}_{n}^{\prime }\left( t\right) dt}\right\vert }\right\} \\
+\dfrac{L_{n}\left( \xi \right) }{2\left\vert u\right\vert ^{2n+3}}\left\{ {1-\dfrac{1}{2\left\vert u\right\vert }\int_{\mathcal{L}\left( \xi \right) }{\left\vert {\psi \left( t\right) dt}\right\vert }}\right\} ^{-1}\int_{\mathcal{L}\left( \xi \right) }{\left\vert {\psi \left( t\right) dt}\right\vert ,}  \label{eq82}
\end{multline}%
in which%
\begin{equation}
L_{n}\left( \xi \right) ={\sup_{w\in \mathcal{L}\left( \xi \right) }}\left\vert {G_{n}\left( w\right) }\right\vert +\frac{1}{2}\int_{\mathcal{L}\left( \xi \right) }{\left\vert {{G}_{n}^{\prime }\left( t\right) dt}\right\vert .}  \label{eq80}
\end{equation}%
In the bounds the integrals and supremum are assumed to exist.

More generally, for any positive integer $r$ 
\begin{multline}
\left\vert {\varepsilon _{n}\left( {u,\xi }\right) }\right\vert \leq
\left\vert\sum\limits_{s=n}^{n+r-1} \dfrac{{G_{s}\left( \xi \right) }}{u^{2s+2}}\right\vert +\dfrac{1}{\left\vert u\right\vert ^{2n+2r+2}}\left\{
\left\vert {G_{n+r}\left( \xi \right) }\right\vert +\dfrac{1}{2}\int_{\mathcal{L}\left( \xi \right) }{\left\vert {{G}_{n+r}^{\prime }\left(
t\right) dt}\right\vert }\right\} \\
+\dfrac{L_{n+r}\left( \xi \right) }{2\left\vert u\right\vert ^{2n+2r+3}}\left\{ {1-\dfrac{1}{2\left\vert u\right\vert }\int_{\mathcal{L}\left( \xi\right) }{\left\vert {\psi \left( t\right) dt}\right\vert }}\right\}^{-1}\int_{\mathcal{L}\left( \xi \right) }{\left\vert {\psi \left( t\right)
dt}\right\vert .}  \label{eq88}
\end{multline}%
\label{thm4}
\end{theorem}

\textbf{Remark}. The bound (\ref{eq88}) is generally sharper than (\ref{eq82}), provided $n+r$ is not too large.

\begin{proof}
Inserting (\ref{eq67}) into (\ref{eq62}), and then using (\ref{eq68}) yields 
\begin{equation}
\frac{d^{2}\varepsilon _{n}}{d\xi ^{2}}-u^{2}\varepsilon _{n}=\psi
\varepsilon _{n}-\frac{G_{n}}{u^{2n}}.  \label{eq69}
\end{equation}%
Then with variation of parameters, we arrive at 
\begin{multline}
\varepsilon _{n}\left( {u,\xi }\right) =\dfrac{e^{-u\xi }}{2u}\int_{\alpha
_{1}}^{\xi }{e^{ut}\left\{ {\dfrac{G_{n}\left( t\right) }{u^{2n}}-\psi \left( t\right) \varepsilon _{n}\left( {u,t}\right) }\right\} dt} \\
+\dfrac{e^{u\xi }}{2u}\int_{\xi }^{\alpha _{2}}{e^{-ut}\left\{ {\dfrac{G_{n}\left( t\right) }{u^{2n}}-\psi \left( t\right) \varepsilon _{n}\left( {u,t}\right) }\right\} dt.}  \label{eq70}
\end{multline}%
The paths of integration in both integrals are chosen to coincide with the
segments of $\mathcal{L}\left( \xi \right) $ from $\alpha _{j}$ to $\xi $.
Next, from integration by parts, we have for the first integral on the RHS
of (\ref{eq70}) 
\begin{equation}
\int_{\alpha _{1}}^{\xi }{e^{ut}G_{n}\left( t\right) dt}=\frac{1}{u}e^{u\xi}G_{n}\left( \xi \right) -\frac{1}{u}\int_{\alpha _{1}}^{\xi }{e^{ut}{G}_{n}^{\prime }\left( t\right) dt,}  \label{eq71}
\end{equation}%
since $\mathrm{Re}\left( u\alpha _{1}\right) =-\infty $. Likewise we have 
\begin{equation}
\int_{\xi }^{\alpha _{2}}{e^{-ut}G_{n}\left( t\right) dt}=\frac{1}{u} e^{-u\xi }G_{n}\left( \xi \right) +\frac{1}{u}\int_{\xi }^{\alpha _{2}}{e^{-ut}{G}_{n}^{\prime }\left( t\right) dt.}  \label{eq72}
\end{equation}%
Consequently we obtain 
\begin{multline}
\varepsilon _{n}\left( {u,\xi }\right) =\dfrac{G_{n}\left( \xi \right) }{u^{2n+2}}-\dfrac{e^{-u\xi }}{2u}\int_{\alpha _{1}}^{\xi }{e^{ut}\left\{ {\dfrac{{G}_{n}^{\prime }\left( t\right) }{u^{2n+1}}+\psi \left( t\right)
\varepsilon _{n}\left( {u,t}\right) }\right\} dt} \\
-\dfrac{e^{u\xi }}{2u}\int_{\xi }^{\alpha _{2}}{e^{-ut}\left\{ {\dfrac{{G}_{n}^{\prime }\left( t\right) }{u^{2n+1}}+\psi \left( t\right) \varepsilon
_{n}\left( {u,t}\right) }\right\} dt,}  \label{eq73}
\end{multline}%
and hence it is evident that 
\begin{equation}
\left\vert {\varepsilon _{n}\left( {u,\xi }\right) }\right\vert \leq
\left\vert \dfrac{{G_{n}\left( \xi \right) }}{u^{2n+2}}\right\vert +\frac{1}{2\left\vert u\right\vert ^{2n+2}}\int_{\mathcal{L}\left( \xi \right) }{\left\vert {{G}_{n}^{\prime }\left( t\right) dt}\right\vert }+\frac{1}{2\left\vert u\right\vert }\int_{\mathcal{L}\left( \xi \right) }{\left\vert {\psi \left( t\right) \varepsilon _{n}\left( {u,t}\right) dt}\right\vert .}
\label{eq75}
\end{equation}%
We now define 
\begin{equation}
M_{n}\left( {u,\xi }\right) ={\sup_{w\in \mathcal{L}\left( \xi \right) }}\left\vert {\varepsilon _{n}\left( {u,w}\right) }\right\vert ,  \label{eq76}
\end{equation}%
which then yields the bound 
\begin{equation}
\int_{\mathcal{L}\left( \xi \right) }{\left\vert {\psi \left( t\right)
\varepsilon _{n}\left( {u,t}\right) dt}\right\vert }\leq M_{n}\left( {u,\xi }\right) \int_{\mathcal{L}\left( \xi \right) }{\left\vert {\psi \left(t\right) dt}\right\vert .}  \label{eq77}
\end{equation}%
If we replace $\xi $ by $w$ in (\ref{eq75}), for any $w\in \mathcal{L}\left(
\xi \right) $, and take the suprema over all such points on this curve, we infer that%
\begin{multline}
{\sup_{w\in \mathcal{L}\left( \xi \right) }}\left\vert {\varepsilon
_{n}\left( {u,w}\right) }\right\vert \leq {\sup_{w\in \mathcal{L}\left(\xi\right) }}\left\vert \dfrac{{G_{n}\left( w \right) }}{u^{2n+2}}\right\vert \\
+\dfrac{1}{2\left\vert u\right\vert ^{2n+2}}\int_{\mathcal{L}\left( \xi
\right) }{\left\vert {{G}_{n}^{\prime }\left( t\right) dt}\right\vert }+ \dfrac{M_{n}\left( {u,\xi }\right) }{2\left\vert u\right\vert }\int_{\mathcal{L}\left( \xi \right) }{\left\vert {\psi \left( t\right) dt}\right\vert .}
\label{eq78}
\end{multline}%
Note that the second and third terms on the RHS are constant along $\mathcal{L}\left( \xi \right) $. Accordingly 
\begin{equation}
M_{n}\left( {u,\xi }\right) \leq \frac{L_{n}\left( \xi \right) }{\left\vert u\right\vert ^{2n+2}}+\frac{M_{n}\left( {u,\xi }\right) }{2\left\vert u\right\vert }\int_{\mathcal{L}\left( \xi \right) }{\left\vert {\psi \left(
t\right) dt}\right\vert ,}  \label{eq79}
\end{equation}%
where $L_{n}\left( \xi \right) $ is defined by (\ref{eq80}). Under the
assumption that $\left\vert u\right\vert $ is sufficiently large so that $\int_{\mathcal{L}\left( \xi \right) }{\left\vert {\psi \left( t\right) dt}\right\vert }<2\left\vert u\right\vert $ then from (\ref{eq79}) we can assert
that 
\begin{equation}
M_{n}\left( {u,\xi }\right) \leq \frac{L_{n}\left( \xi \right) }{\left\vert
u\right\vert ^{2n+2}}\left\{ {1-\frac{1}{2\left\vert u\right\vert }\int_{\mathcal{L}\left( \xi \right) }{\left\vert {\psi \left( t\right) dt}\right\vert }}\right\} ^{-1}.  \label{eq81}
\end{equation}%
From this bound, along with (\ref{eq75}), (\ref{eq77}) and (\ref{eq81},) we conclude that (\ref{eq82}) holds.

This error bound establishes that $G\left( {u,\xi }\right) $ is bounded at both $\xi =\alpha _{1}$ and $\xi =\alpha _{2}$. Consider now any other
particular solution with this property. We know it can be expressed in the form (\ref{eq66}). However, if we let $\xi \rightarrow \alpha _{1}$ ($\mathrm{Re}\left( u\xi \right) \rightarrow -\infty $) in this expression we
find that $B\left( u\right) $ must be identically zero, otherwise the solution would be unbounded in this limit. Likewise, $A\left( u\right) $
must be identically zero, otherwise the solution would be unbounded at $\xi=\alpha _{2}$. This verifies that $G\left( {u,\xi }\right) $ is the unique
solution satisfying (\ref{eq67}) - (\ref{eq80}).

Finally, for any positive integer $r$, we immediately deduce from (\ref{eq67}) and uniqueness of the solution that 
\begin{equation}
\varepsilon _{n}\left( {u,\xi }\right) =\sum\limits_{s=n}^{n+r-1}\frac{G_{s}\left( \xi \right) }{u^{2s+2}}+\varepsilon _{n+r}\left( {u,\xi }\right)
.  \label{eq87}
\end{equation}%
Hence from this, and with $n$ replaced by $n+r_{\,}$in (\ref{eq82}), we arrive at the alternative bound (\ref{eq88}).
\end{proof}

\section{Modified Bessel functions}

\label{sec5}

We now apply the expansions of theorem \ref{thm2} to modified Bessel functions $I_{\nu }\left( {z}\right) $ and $K_{\nu }\left( {z}\right) $. To
this end we observe from \cite[chap. 10, eq. (7.02)]{Olver:1997} that $z^{1/2}I_{\nu }\left( {\nu z}\right) $ and $z^{1/2}K_{\nu }\left( {\nu z}\right) $ satisfies the differential equation (\ref{eq1}) with $u=\nu $ (which we assume to be real and positive), and 
\begin{equation}
f\left( z\right) =\frac{1+z^{2}}{z^{2}},\ g\left( z\right) =-\frac{1}{4z^{2}}.
\label{eq90}
\end{equation}%
We note that $f\left( z\right) $ and $g\left(z\right) $ meet the requirements of (\ref{eq19a}) and (\ref{eq19b}), thus ensuring that the
expansions we shall derive will be uniformly valid at both singularities $z=0$ and $z=\infty $.

From (\ref{eq3}) we obtain the new independent variable explicitly as 
\begin{equation}
\xi =\int {f^{1/2}\left( z\right) dz}=\left( {1+z^{2}}\right) ^{1/2}+\ln
\left\{ {\frac{z}{1+\left( {1+z^{2}}\right) ^{1/2}}}\right\} .  \label{eq91}
\end{equation}%
A full description of the $z-\xi $ map in the complex plane is given in \cite[chap. 10, section 7]{Olver:1997}, and so it is not necessary for us to give details, except to note that $z=0$ corresponds to $\xi =-\infty $, and $z=+\infty $ corresponds to $\xi =+\infty $.

From (\ref{eq91}), along with the new dependent variable given by (\ref{eq4}), we obtain the transformed equation (\ref{eq5}) where 
\begin{equation}
\psi \left( \xi \right) =\frac{z^{2}\left( {4-z^{2}}\right) }{4\left( {1+z^{2}}\right) ^{3}}.  \label{eq92}
\end{equation}%
It is convenient to work with the variable 
\begin{equation}
p=\left( {1+z^{2}}\right) ^{-1/2}.  \label{eq93}
\end{equation}%
Then, from (\ref{eq2}), (\ref{eq11}), (\ref{eq13}) and (\ref{eq93}), and
using the notation $\tilde{{F}}_{s}\left( p\right) =\hat{F}_{s}\left(z \left(p\right) \right) $, we obtain the coefficients for this case as being given by 
\begin{equation}
\tilde{{F}}_{1}\left( p\right) ={\tfrac{1}{8}}p^{2}\left( {1-p^{2}}\right)
\left( {5p^{2}-1}\right) ,  \label{eq94}
\end{equation}
\begin{equation}
\tilde{{F}}_{2}\left( p\right) ={\tfrac{1}{8}}\,p^{3}\left( {1-p^{2}}\right)
\left( {12\,p^{2}-15\,p^{4}-1}\right),  \label{eq95}
\end{equation}%
and 
\begin{equation}
\tilde{{F}}_{s+1}\left( p\right) =\frac{1}{2}p^{2}\left( {1-p^{2}}\right) 
\frac{d\tilde{{F}}_{s}\left( p\right) }{dp}-\frac{1}{2}\sum\limits_{j=1}^{s-1}{\tilde{{F}}_{j}\left( p\right) \tilde{{F}}_{s-j}\left(
p\right) }\ \left( {s=2,3,4,\cdots }\right) .  \label{eq96}
\end{equation}%
Therefore from (\ref{eq2}), (\ref{eq10}), and (\ref{eq93}) we have 
\begin{equation}
\tilde{{E}}_{s}\left( p\right) =-\int_{0}^{p}{\frac{\tilde{{F}}_{s}\left(q\right) }{q^{2}\left( {1-q^{2}}\right) }dq}\ \left( {s=1,2,3,\cdots }\right),  \label{eq97}
\end{equation}
where $\tilde{{E}}_{s}\left( p\right) =\hat{E}_{s}\left(z \left(p\right) \right) $.
The lower limits of integration are chosen here for convenience only.

Now from (\ref{eq94}) - (\ref{eq96}) we see by induction that $p^{2}\left( {1-p^{2}}\right) $ is a factor for each $\tilde{{F}}_{s}\left( p\right) $, and hence from (\ref{eq97}) each $\tilde{{E}}_{s}\left( p\right) $ is also a
polynomial in $p$. The first two are given by 
\begin{equation}
\tilde{{E}}_{1}\left( p\right) ={\tfrac{1}{{24}}}p\,\left( {3-5\,p^{2}}\right) ,\ \tilde{{E}}_{2}\left( p\right) ={\tfrac{1}{{16}}}\,p^{2}\left( {1-p^{2}}\right) \left( {1-5\,p^{2}}\right) .  \label{eq98}
\end{equation}

Noting that $p=1$ corresponds to $z=0$ we next define 
\begin{equation}
\tilde{{E}}_{s}\left( 1\right) =k_{s}\ \left( {s=1,2,3,\cdots }\right) .
\label{eq99}
\end{equation}%
It is straightforward to verify that $k_{2m}=0\ \left( {m=1,2,3,\cdots }\right) $, and the first three non-zero (odd terms) being given by 
\begin{equation}
k_{1}=-{\tfrac{1}{{12}}},\ k_{3}={\tfrac{1}{{360}}},\ k_{5}=-{\tfrac{1}{{1260}}.}  \label{eq100}
\end{equation}

We can now match solutions which are recessive at the singularities $z=0$
and $z=+\infty $, and to this end we apply theorem \ref{thm2} with $\alpha
_{1}=-\infty $ and $\alpha _{2}=\infty $. For the solutions that are recessive
at $z=0$ ($\xi =-\infty $) we have from (\ref{eq4}) and (\ref{eq90}) 
\begin{equation}
I_{\nu }\left( {\nu z}\right) =\frac{c_{1}\left( \nu \right) }{\left( {1+z^{2}}\right) ^{1/4}}\exp \left\{ {\sum\limits_{s=1}^{n-1}{\frac{\tilde{{E}}_{s}\left( p\right) -k_{s}}{\nu ^{s}}}}\right\} \left\{ {e^{\nu \xi}+\varepsilon _{n,1}\left( {\nu ,\xi }\right) }\right\} ,  \label{eq101}
\end{equation}%
for some constant $c_{1}\left( \nu \right) $. From (\ref{eq91}) we find as $z\rightarrow 0$ 
\begin{equation}
\xi =\ln \left( {{\tfrac{1}{2}}z}\right) +1+\mathcal{O}\left( {z^{2}}\right)
,  \label{eq102}
\end{equation}%
and we know that in the same circumstances (\cite[chap. 12, eq. (1.01)]{Olver:1997})%
\begin{equation}
I_{\nu }\left( {\nu z}\right) =\frac{\left( {{\frac{1}{2}}\nu z}\right)
^{\nu }}{\Gamma \left( {\nu +1}\right) }\left\{ {1+\mathcal{O}\left(z\right) }\right\} .  \label{eq103}
\end{equation}%
On using (\ref{eq93}), (\ref{eq99}), (\ref{eq101}) and (\ref{eq102}) we determine that 
\begin{equation}
c_{1}\left( \nu \right) =\frac{\nu ^{\nu }}{e^{\nu }\Gamma \left( {\nu +1}\right) }.  \label{eq104}
\end{equation}

The matching of solutions that are recessive at $z=+\infty $ ($\xi =+\infty $) is similarly established. We have the relation 
\begin{equation}
K_{\nu }\left( {\nu z}\right) =\frac{c_{2}\left( \nu \right) }{\left( {1+z^{2}}\right) ^{1/4}}\exp \left\{ {\sum\limits_{s=1}^{n-1}{\left( {-1}\right) ^{s}\frac{\tilde{{E}}_{s}\left( p\right) }{\nu ^{s}}}}\right\}\left\{ {e^{-\nu \xi }+\varepsilon _{n,2}\left( {\nu ,\xi }\right) }\right\}
,  \label{eq105}
\end{equation}%
and on using 
\begin{equation}
K_{\nu }\left( {\nu z}\right) =\left( {\frac{\pi }{2\nu z}}\right)^{1/2}e^{-\nu z}\left\{ {1+\mathcal{O}\left( {\frac{1}{z}}\right) }\right\} ,
\label{eq106}
\end{equation}%
as $z\rightarrow +\infty $, along with $\xi =z+\mathcal{O}\left( {z^{-1}}\right) $ which is verifiable from (\ref{eq91}), we obtain the desired
formula 
\begin{equation}
c_{2}\left( \nu \right) =\left( {\frac{\pi }{2\nu }}\right) ^{1/2}.
\label{eq107}
\end{equation}%
In this derivation we used the fact that $p=0$ corresponds to $z=+\infty $, and from (\ref{eq97})$\tilde{\text{ }{E}}_{s}\left( 0\right) =0\ \left( {s=1,2,3,\cdots }\right) $.

Let us now examine the error bounds in more detail. To this end, combining (\ref{eq101}) and (\ref{eq104}) yields
\begin{equation}
I_{\nu }\left( {\nu z}\right) =\frac{\nu ^{\nu }}{e^{\nu }\Gamma \left( {\nu+1}\right) \left( {1+z^{2}}\right) ^{1/4}}\exp \left\{ {\nu \xi
+\sum\limits_{s=1}^{n-1}{\frac{\tilde{{E}}_{s}\left( p\right) -k_{s}}{\nu
^{s}}}}\right\} \left\{ {1+\eta _{n,1}\left( {\nu ,z}\right) }\right\} ,
\label{eq108}
\end{equation}%
where 
\begin{equation}
\eta _{n,1}\left( {\nu ,z}\right) =e^{-\nu \xi }\varepsilon _{n,1}\left( {\nu ,\xi }\right) .  \label{eq109}
\end{equation}%
Thus from (\ref{eq23}), (\ref{eq16}) - (\ref{eq19}) and theorem \ref{thm2}
we have the error bound 
\begin{multline}
\left\vert {\eta _{n,1}\left( {\nu ,z}\right) }\right\vert \leq \left\vert {\exp \left\{ {\sum\limits_{s=n}^{n+r-1}{\dfrac{\tilde{{E}}_{s}\left(
p\right) -k_{s}}{\nu ^{s}}}}\right\} -1}\right\vert \\
+\dfrac{\omega _{n+r,1}\left( {\nu ,p}\right) }{\nu ^{n+r}}\exp \left\{ {\dfrac{\varpi _{n+r,1}\left( {\nu ,p}\right) }{\nu }+\sum\limits_{s=n}^{n+r-1}{\dfrac{\mathrm{Re}\tilde{{E}}_{s}\left( p\right) -k_{s}}{\nu ^{s}}}+\dfrac{\omega _{n+r,1}\left( {\nu ,p}\right) }{\nu ^{n+r}}}\right\} ,  \label{eq113}
\end{multline}%
where 
\begin{equation}
\omega _{n,1}\left( {\nu ,p}\right) =2\int_{1}^{p}{\left\vert {\frac{\tilde{{F}}_{n}\left( q\right) dq}{q^{2}\left( {1-q^{2}}\right) }}\right\vert }+\sum\limits_{s=1}^{n-1}{\frac{1}{\nu ^{s}}\int_{1}^{p}{\left\vert {\frac{\tilde{{G}}_{n,s}\left( q\right) dq}{q^{2}\left( {1-q^{2}}\right) }}\right\vert ,}}  \label{eq114}
\end{equation}%
in which 
\begin{equation}
\tilde{{G}}_{n,s}\left( p\right) =\sum\limits_{k=s}^{n-1}{\tilde{{F}}_{k}\left( p\right) \tilde{{F}}_{s+n-k-1}\left( p\right) ,}  \label{eq115}
\end{equation}%
and 
\begin{equation}
\varpi _{n,1}\left( {\nu ,p}\right) =\sum\limits_{s=0}^{n-2}{\frac{4}{\nu
^{s}}\int_{1}^{p}{\left\vert {\frac{\tilde{{F}}_{s+1}\left( q\right) dq}{q^{2}\left( {1-q^{2}}\right) }}\right\vert .}}  \label{eq116}
\end{equation}

The bound (\ref{eq113}) holds uniformly in an unbounded $z$-domain that includes the half plane $\left\vert \arg \left( z\right) \right\vert \leq 
\frac{1}{2}\pi $, excluding points on and near the imaginary axis from $z=i$ to $z=i\infty $, and\ from $z=-i$ to $z=-i\infty $.

Similarly, one can show that 
\begin{equation}
K_{\nu }\left( {\nu z}\right) =\left( {\dfrac{\pi }{2\nu }}\right) ^{1/2}\dfrac{1}{\left( {1+z^{2}}\right) ^{1/4}}\exp \left\{ -{\nu \xi
+\sum\limits_{s=1}^{n-1}}\left( -1\right) ^{s}{{\dfrac{\tilde{{E}}_{s}\left(p\right) }{\nu ^{s}}}}\right\} \left\{ {1+\eta _{n,2}\left( {\nu ,z}\right) }\right\} ,  \label{eq116a}
\end{equation}%
where%
\begin{multline}
\left\vert {\eta _{n,2}\left( {\nu ,z}\right) }\right\vert \leq \left\vert {\exp \left\{ {\sum\limits_{s=n}^{n+r-1}\left( -1\right) ^{s}{\dfrac{\tilde{{E}}_{s}\left( p\right) }{\nu ^{s}}}}\right\} -1}\right\vert  \\
+\dfrac{\omega _{n+r,2}\left( {\nu ,p}\right) }{\nu ^{n+r}}\exp \left\{ {\dfrac{\varpi _{n+r,2}\left( {\nu ,p}\right) }{\nu }+\sum\limits_{s=n}^{n+r-1}\left( -1\right) ^{s}{\dfrac{\mathrm{Re}\tilde{{E}}_{s}\left( p\right) }{\nu ^{s}}}+\dfrac{\omega _{n+r,2}\left( {\nu ,p}\right) }{\nu ^{n+r}}}\right\} .  \label{eq116b}
\end{multline}%
Here $\omega _{n,2}\left( {\nu ,p}\right) $ and $\varpi _{n+r,2}\left( {\nu,p}\right) $ are given by (\ref{eq114}) and (\ref{eq116}) respectively,
except the integrals are taken along progressive paths from $q=0$ to $q=p$.
The bound (\ref{eq116b}) holds uniformly in an unbounded $z$-domain that
includes the half plane $\left\vert \arg \left( z\right) \right\vert \leq 
\frac{1}{2}\pi $, excluding a neighbourhood of the points $z=\pm i$.

\begin{table}[H]
\begin{equation*}
\begin{array}{|c|c|c|}
\hline
\boldmath z\  & \boldmath \left\vert {\eta _{5,1}\left( {20 ,z}\right) }\right\vert & \text{Error bound} \\ \hline
0.01 & \num{7.418601e-12} & \num{7.418606e-12} \\ \hline
0.1 & \num{5.422462e-10} & \num{5.422471e-10} \\ \hline
1 & \num{6.1812e-9} & \num{6.1822e-9} \\ \hline
10 & \num{2.470e-10} & \num{2.493e-10} \\ \hline
100 & \num{2.476e-10} & \num{2.488e-10} \\ \hline
\end{array}
\end{equation*}
\caption{\textit{Exact relative error, and bounds from (\protect\ref{eq113})
with $\protect\nu=20$ and $n=r=5$}}
\label{table1}
\end{table}

Table \ref{table1} compares the bound (\ref{eq113}) with $\nu=20$ and $n=r=5$, to the absolute value of exact relative error $\left\vert {\eta_{5,1}\left( {20 ,z}\right) }\right\vert$ for various values of $z$. We see that the bounds are sharp for all values of $z$, and that the expansion (\ref{eq108}) provides a good approximation uniformly for all $z$. Note that for our values $\nu=20$ and $n=5$ we have $\nu^{-n}=\num{3.125e-7}$.

For $r=0$ we have the unmodified bound that comes from theorem \ref{thm1}, namely%
\begin{equation}
\left\vert {\eta _{n,1}\left( {\nu ,z}\right) }\right\vert \leq \dfrac{\omega _{n,1}\left( {\nu ,p}\right) }{\nu ^{n}}\exp \left\{ {\dfrac{\varpi
_{n,1}\left( {\nu ,p}\right) }{\nu }+\dfrac{\omega _{n,1}\left( {\nu ,p}\right) }{\nu ^{n}}}\right\} .  \label{eq120}
\end{equation}%
Table \ref{table2} compares this bound to the exact error for the same values of $z$, ${\nu }$, and $n$ as table \ref{table1}. We see that the bound is fairly sharp close to $z=0$ ($p=1$), but overestimates the true error by an order of magnitude for other values of $z$. In essence the unmodified bound is of the expected $\nu$ order of magnitude, i.e. it is $\mathcal{O}\left(\nu^{-n}\right)$ uniformly for $0\leq z< \infty$, but fails to capture
how small the true error is.

\begin{table}[H]
\begin{equation*}
\begin{array}{|c|c|c|}
\hline
\boldmath z\  & \boldmath \left\vert {\eta _{5,1}\left( {20 ,z}\right) }\right\vert & \text{Error bound} \\ \hline
0.01 & \num{7.41e-12} & \num{1.56e-11} \\ \hline
0.1 & \num{5.42e-10} & \num{1.12e-9} \\ \hline
1 & \num{6.18e-9} & \num{4.15e-8} \\ \hline
10 & \num{2.47e-10} & \num{5.75e-8} \\ \hline
100 & \num{2.48e-10} & \num{5.76e-8} \\ \hline
\end{array}
\end{equation*}
\caption{\textit{Exact relative error, and bounds from (\protect\ref{eq120})
with $\protect\nu=20$ and $n=5$}}
\label{table2}
\end{table}

\begin{figure}[!htb]
\vspace*{0.8cm} \centerline{\includegraphics[height=6cm,width=8cm]{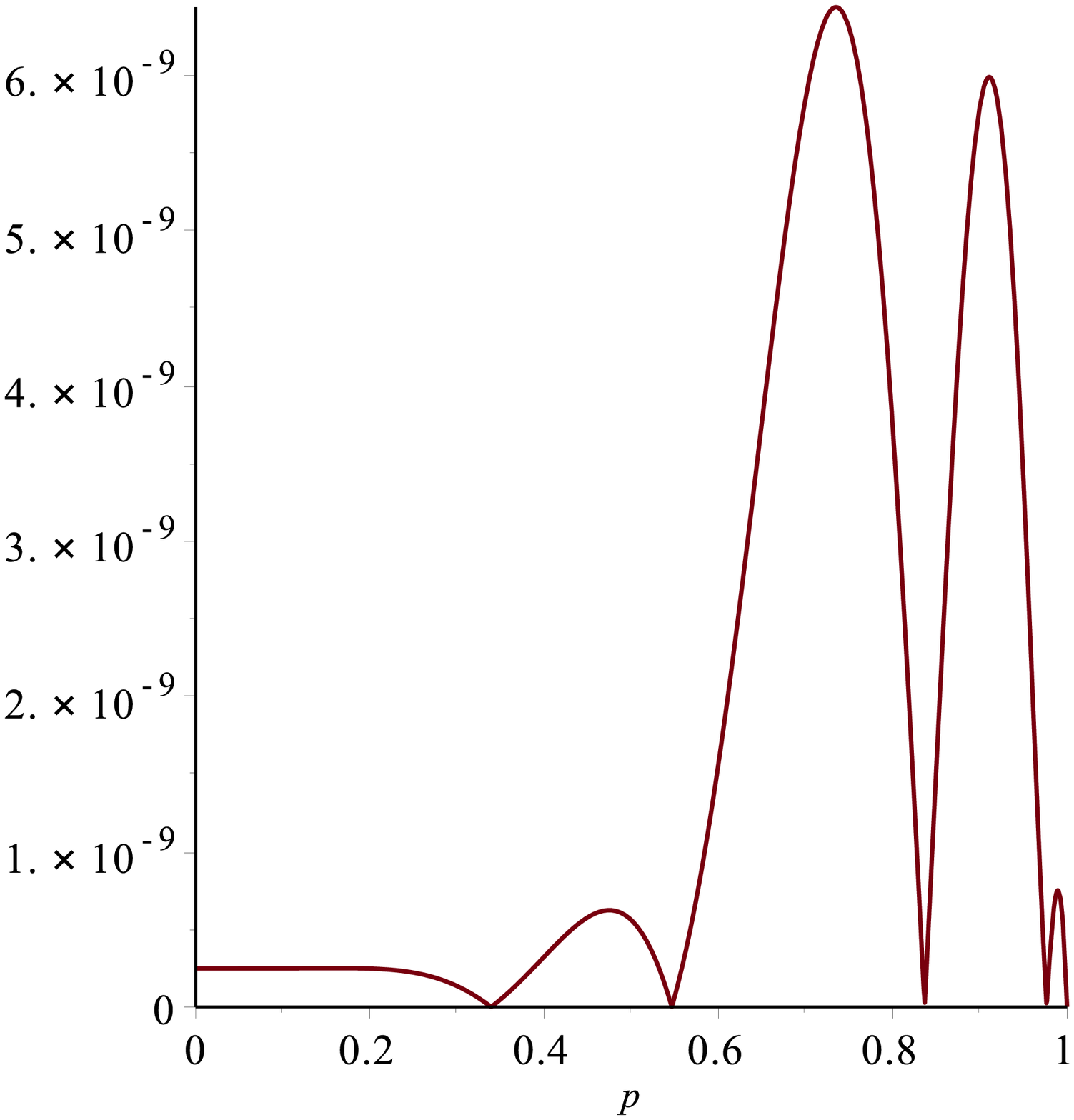}}
\caption{\textit{Plot of $\Phi _{5}\left( 20,p\right) $}}
\label{fig1}
\end{figure}

To see why (\ref{eq120}) is not so sharp (and indeed this is true of most large parameter error estimates coming from successive approximations), consider first the absolute value of the exact relative error $\left\vert {\eta _{n,1}\left( {\nu ,z}\right) }\right\vert $: we expect this to be close
to the first neglected term of the expansion (\ref{eq108}). Now for large real $\nu $ this term is approximately $\Phi _{n}\left( \nu ,p\right) $,
where 
\begin{equation}
\Phi _{n}\left( \nu ,p\right) := \dfrac{\left\vert\tilde{{E}}_{n}\left(
p\right) -k_{n}\right\vert}{\nu ^{n}}.  \label{eq121}
\end{equation}%
Figure \ref{fig1} depicts the graph of $\Phi _{5}\left( 20,p\right) $ for $0\leq p\leq 1$. We observe that it is relatively small for $0\leq
p\lessapprox 0.6$ as well as for $p$ close to $1$, and of course near its zeros.

On the other hand, for large $\nu $ the error bound (\ref{eq120}) is approximated by the term outside the exponential, and further $\omega_{n,1}\left( {\nu ,p}\right)$ is approximated by the leading term in (\ref{eq114}). Thus the error bound (\ref{eq120}) is approximated by the function 
\begin{equation}
\Omega _{n}\left( \nu ,p\right) :=\frac{2}{\nu ^{n}}\int_{p}^{1}\frac{{\left\vert \tilde{{F}}_{n}\left( q\right) \right\vert }dq}{q^{2}\left( {1-q^{2}}\right) }.  \label{eq122}
\end{equation}

\begin{figure}[!htb]
\vspace*{0.8cm} \centerline{\includegraphics[height=6cm,width=8cm]{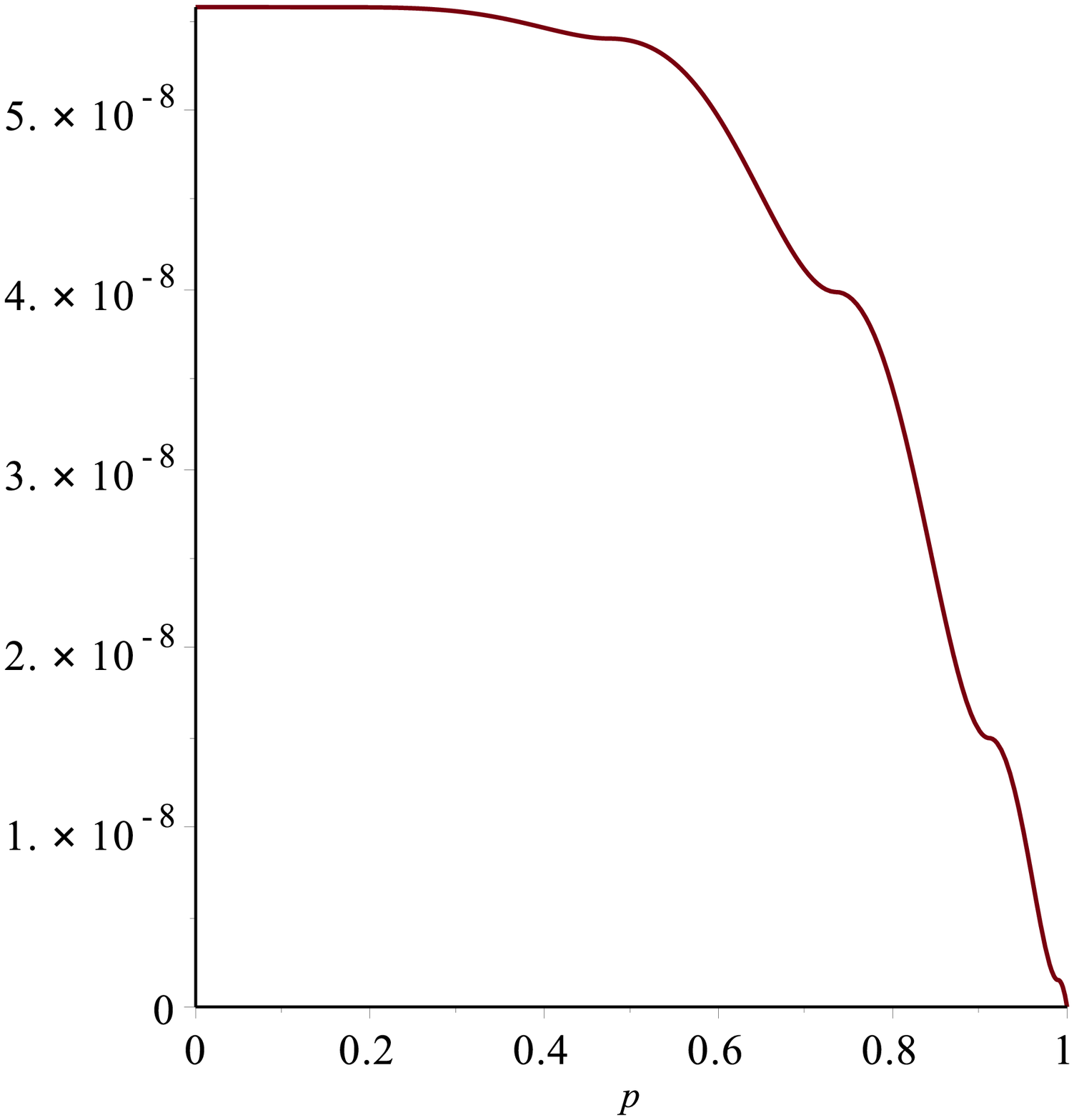}}
\caption{\textit{Plot of $\Omega _{5}\left( 20,p\right)$}}
\label{fig2}
\end{figure}

Figure \ref{fig2} depicts the graph of $\Omega _{5}\left( 20,p\right)$ for $0\leq p\leq 1$, which is obviously monotonically increasing as $p$ decreases
from $1$ to $0$. On comparison to figure \ref{fig1} it is clear that it overestimates the true error if $p$ is not close to $1$, and this discrepancy is exacerbated near the zeros of $\Phi _{n}\left( \nu ,p\right) $ (where the exact error can be expected to be ${\mathcal{O}\left( \nu^{-n-1}\right) }$).

This is highlighted in figure \ref{fig3}, where the ratio $\Phi _{5}\left(
20,p\right) /\Omega _{5}\left( 20,p\right) $ is plotted for $0\leq p\leq 1$,
and is shown to be small for $0\leq p\lessapprox 0.6$, as well as in the
neighbourhood of the zeros of $\Phi _{n}\left( \nu ,p\right) $.

\begin{figure}[ptb]
\vspace*{0.8cm} \centerline{\includegraphics[height=6cm,width=8cm]{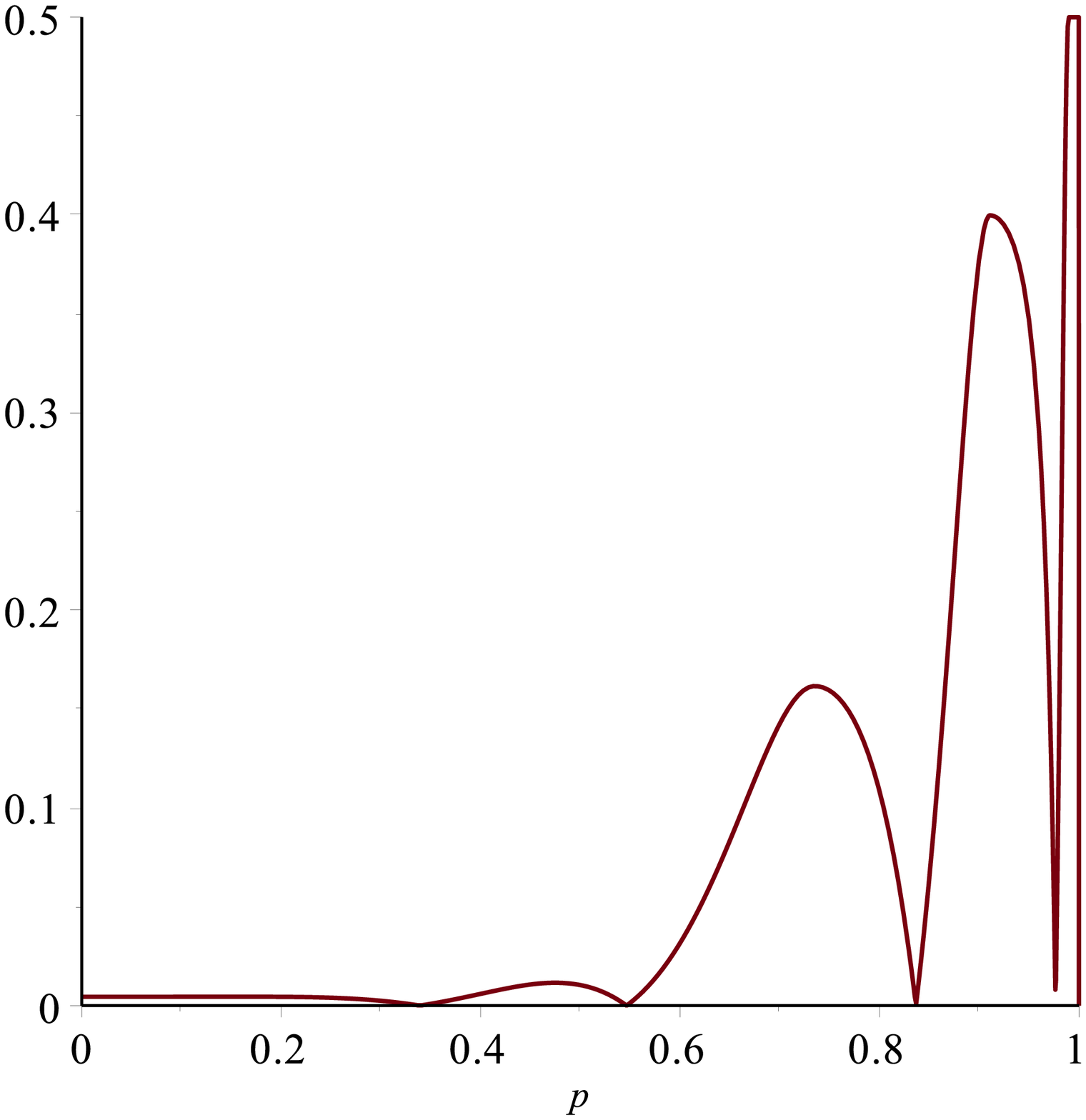}}
\caption{\textit{Plot of $\Phi _{5}\left( 20,p\right) /\Omega _{5}\left(
20,p\right)$ }}
\label{fig3}
\end{figure}

We mention that the choice of $r$ in our modified bounds (\ref{eq113}) and (\ref{eq116b}) (and more generally in theorem \ref{thm2}) is not completely arbitrary, due to the divergent nature of the asymptotic expansions. Specifically, $n+r$ should not be larger than the number of terms ($n_{0}$ say) that give maximal accuracy. In general it is usual for $n_{0}=\mathcal{O}\left( u\right) $, and if we choose to take this number of terms in an expansion, then we must resort to using the unmodified error bounds of theorem \ref{thm1}.

We conclude by mentioning that for large $z$ a further improvement in accuracy in the asymptotic approximation for $I_{\nu }\left( \nu z\right) $ may be possible via a connection formula, and likewise for $K_{\nu }\left( \nu z\right) $ near $z=0$. To illustrate this for the former modified Bessel function, we use the
relation (\cite[chap. 7, ex. 8.2]{Olver:1997})  
\begin{equation}
\pi iI_{\nu }\left( \nu z\right) =K_{\nu }\left( \nu ze^{-\pi i}\right)
-e^{\nu \pi i}K_{\nu }\left( \nu z\right) ,  \label{eq117}
\end{equation}%
along with asymptotic expansions for $K_{\nu }\left( \nu z\right) $ (given
by (\ref{eq116a})), and a similar expansion for $K_{\nu }\left( \nu ze^{-\pi
i}\right) ,$ namely 
\begin{equation}
K_{\nu }\left( {\nu ze}^{-\pi i}\right) =i\left( {\frac{\pi }{2\nu }}\right)
^{1/2}\frac{1}{\left( {1+z^{2}}\right) ^{1/4}}\exp \left\{ {\nu \xi +
\sum\limits_{s=1}^{n-1}{\frac{\tilde{{E}}_{s}\left( p\right) }{\nu ^{s}}}}\right\} \left\{ {1+\eta_{n,3}\left( {\nu ,\xi }\right) } \right\} .  \label{eq119}
\end{equation}

We omit details on the derivation of this expansion, except to remark that it is based on the recessive behaviour of the function at $z=\infty e^{\pi i},$ and is uniformly valid in an unbounded domain which includes the half-plane $\mathrm{Re}z>z_{0}$ ($\left\vert \arg z\right\vert <\frac{1}{2}\pi $), where $z_{0}=0.66274\cdots $ is the point on the positive real axis labeled by $\mathsf{B}$ in \cite[chap. 10, fig. 7.1]{Olver:1997}. In the expansion (\ref{eq119}) the error term ${\eta_{n,3}\left( {\nu ,\xi }\right) }$ has the same bound (\ref{eq116b}) as ${\eta_{n,2}\left( {\nu ,\xi }\right) }$,
except the path of integration must meet different monotonicity requirements; one acceptable path is parameterised by $q=\left( {1+}\left(
z+i\tau \right) {^{2}}\right) ^{-1/2}$ ($0\leq \tau <\infty $). Importantly, ${\eta_{n,3}\left( {\nu ,\xi }\right) }$ vanishes as $\mathrm{Re}z\rightarrow +\infty $, as does ${\eta_{n,2}\left( {\nu ,\xi }\right) }$. Thus on inserting (\ref{eq116a}) and (\ref{eq119}) into (\ref{eq117}) we obtain a compound asymptotic expansion for $I_{\nu }\left( \nu z\right) $ whose accuracy improves as $\mathrm{Re}z\rightarrow +\infty $, in contrast to (\ref{eq108}).

\end{document}